\documentclass[onefignum,onetabnum]{siamart171218}

\usepackage{amsfonts}
\usepackage{amssymb}
\usepackage{graphicx}
\usepackage{epstopdf}
\usepackage{algorithmic}
\ifpdf
  \DeclareGraphicsExtensions{.eps,.pdf,.png,.jpg}
\else
  \DeclareGraphicsExtensions{.eps}
\fi

\usepackage{bm}
\usepackage{tikz-cd}
\usepackage{subfig}
\usepackage{mathtools}
\usepackage{stmaryrd}
\usepackage{lipsum}
\newsiamremark{remark}{Remark}
\newsiamremark{assumption}{Assumption}

\newsiamremark{hypothesis}{Hypothesis}
\crefname{hypothesis}{Hypothesis}{Hypotheses}
\newsiamthm{claim}{Claim}

\def\vavg#1{\{\!\!\{#1\}\!\!\}}
\newcommand{\triplenorm}[1]{%
	\left\vert\kern-0.9pt\left\vert\kern-0.9pt\left\vert #1
	\right\vert\kern-0.9pt\right\vert\kern-0.9pt\right\vert}
\headers{DG FOR CURL ADVECTION-DIFFUSION}{J. Wang and S. Wu}

\title{Discontinuous Galerkin Methods for Magnetic Advection-Diffusion Problems
\thanks{%Submitted to the editors DATE.  
The work of Shuonan Wu is supported in part by the National Natural
Science Foundation of China grant No. 11901016.}
}

\author{
Jindong Wang\thanks{School of Mathematical Sciences, Peking
University, Beijing 100871, China} (\email{jdwang@pku.edu.cn}). \and 
Shuonan Wu\thanks{School of Mathematical Sciences, Peking University,
Beijing 100871, China (\email{snwu@math.pku.edu.cn}, 
\url{http://dsec.pku.edu.cn/\~snwu}).}
}

\usepackage{amsopn}

\ifpdf
\hypersetup{
  pdftitle={Primal CURL advection-diffusion},
  pdfauthor={J. Wang and S. Wu}
}
\fi

\begin{document}

\maketitle

% REQUIRED
%\begin{abstract}
%We construct a family of primal discontinuous Galerkin methods for the magnetic advection-diffusion problems based on the weighted-residual approach. 
%  We relax the usual condition for the Friedrichs symmetric operators in the stability analysis, which 
%\end{abstract}
\begin{abstract}
We devise and analyze a class of the primal discontinuous Galerkin methods for the magnetic advection-diffusion problems based on the weighted-residual approach. In addition to the upwind stabilization, we find a new mechanism under the vector case that provides more flexibility in constructing the schemes. 
For the more general Friedrichs system, we show the stability and optimal error estimate, which boil down to two core ingredients -- the weight function and the special projection -- that contain information of advection. Numerical experiments are provided to verify the theoretical results.
\end{abstract}

% REQUIRED

% REQUIRED

\section{Introduction} \label{sc:intro}

Let $\bm{f} \in \bm{L}^2(\Omega)$, $\bm{g}_D \in \bm{H}^{3/2}(\Gamma_D)$, $\bm{g}_N \in \bm{H}^{1/2}(\Gamma_N)$. We consider the following magnetic advection-diffusion problem: 
\begin{equation} \label{eq:Hcurl-cd}
\left\{
\begin{aligned}
\nabla \times (\varepsilon \nabla \times \bm{u} ) - \bm{\beta} \times
  (\nabla \times \bm{u}) + \nabla (\bm{\beta} \cdot \bm{u}) +  \gamma
  \bm{u}&= \bm{f} ~~\quad \text{in }\Omega, \\
\bm{n} \times \bm{u} + \chi_{\Gamma_D^-} (\bm{u} \cdot \bm{n}) \bm{n}
  &= \bm{g}_D \quad \text{on } \Gamma_D, \\
\varepsilon \bm{n} \times (\nabla \times \bm{u}) + \chi_{\Gamma_N^-}
  (\bm{\beta} \cdot \bm{n}) \bm{u} &= \bm{g}_N \quad \text{on }
  \Gamma_N,
\end{aligned}
\right.
\end{equation}
where $\Omega \subset \mathbb{R}^3$ is a bounded, convex, polygonal domain, $\Gamma_D\neq \emptyset$ and $\Gamma_N$ are the parts of the boundary $\Gamma=\partial\Omega$ where Dirichlet and Neumann boundary
conditions are assigned. The velocity field is assumed to be $\bm{\beta}({x}) \in \bm{W}^{1,\infty}(\Omega)$.  $\gamma({x}) \in L^\infty(\Omega)$ is the reaction coefficient, and $\varepsilon$ is a positive constant diffusivity coefficient. The magnetic advection is known as the Lie derivative
\begin{equation} \label{eq:Lie-advection}
L_{\bm{\beta}} \bm{u} := - \bm{\beta} \times (\nabla \times \bm{u}) +
  \nabla (\bm{\beta} \cdot \bm{u}).
\end{equation}
The inflow and outflow parts of $\Gamma = \partial \Omega$ are defined in the usual fashion:
$$ 
\Gamma^- := \{x \in \Gamma: ~ \bm{\beta}(x) \cdot \bm{n}(x) < 0\}, \quad 
\Gamma^+ := \{x \in \Gamma: ~ \bm{\beta}(x) \cdot \bm{n}(x) \geq 0\},
$$ 
where $\bm{n}(x)$ denotes the unit outward normal vector to $\Gamma$
at $x \in \Gamma$. We further denote $\Gamma_D^\pm := \Gamma_D \cap \Gamma^\pm$,
$\Gamma_N^\pm := \Gamma_N \cap \Gamma^\pm$.

The model problem \eqref{eq:Hcurl-cd} arises in many important applications, especially the magnetohydrodynamics \cite{gerbeau2006mathematical}. For instance, it can model the behavior of electromagnetic fields in the stationary flow field of a conducting fluid \cite{heumann2011eulerian, heumann2013stabilized}, where $\bm{u}$ stands for the magnetic vector potential.

One of the main numerical challenges for \eqref{eq:Hcurl-cd} is the stability with respect to $\varepsilon$ for the {\it advection-dominated} case, i.e., $\varepsilon \ll |\bm{\beta}|$. Such an issue becomes more apparent and well-studied for the closely related scalar advection-diffusion problem 
\begin{equation} \label{eq:scalar-cd}
-\nabla\cdot (\varepsilon \nabla u) + \bm{\beta}\cdot \nabla u+\gamma u=f \quad \text{in } \Omega,
\end{equation}
as the standard finite element methods for \eqref{eq:scalar-cd} may suffer from strong numerical oscillations and instability due to the presence of interior or boundary layer. To solve this challenge for the scalar case, one class of methods is based on the idea of exponential fitting, where some exponential functions are introduced in either constructing basis functions \cite{o1991analysis,o1991globally,dorfler1999uniform,dorfler1999uniform2,wang1997novel} or assembling stiffness matrix \cite{xu1999monotone}. Another class of methods for \eqref{eq:scalar-cd} follows the idea of upwind or streamline stabilization. Such methods include the stabilized discontinuous Galerkin method \cite{houston2002discontinuous, brezzi2004discontinuous,ayuso2009discontinuous}, streamline upwind/Petrov-Galerkin (SUPG) method \cite{mizukami1985petrov,hughes1979finite, franca1992stabilized, chen2005optimal, burman2010consistent}, Galerkin/least squares finite element method \cite{hughes1989new}, bubble function stabilization \cite{brezzi1994choosing, brezzi1998applications, brezzi1998further, brezzi1999priori, franca2002stability}, local projection stabilization \cite{braack2006local,matthies2007unified}, edge stabilization \cite{burman2004edge}, and continuous interior penalty method \cite{burman2005unified, burman2007continuous,burman2009weighted}. 

In recent years, the vector case of the advection-diffusion problem \eqref{eq:Hcurl-cd} has received more and more attention from the numerical computation. In terms of exterior calculus, \eqref{eq:scalar-cd} and \eqref{eq:Hcurl-cd} are basically particular instances of a linear advection-diffusion problem for differential forms \cite{heumann2011eulerian, heumann2013stabilized}. Therefore,  the underlying philosophy of devising their schemes would have something in common. In fact, motivated by \cite{xu1999monotone}, a simplex-averaged finite element (SAFE) method for both scalar and vector cases of advection-diffusion problems was proposed in \cite{wu2020simplex}, which is of the class of exponential fitting methods that lead to a natural upwind effect when the diffusion coefficient approaches to zero. The stabilized Galerkin method was studied in \cite{heumann2013stabilized} for \eqref{eq:Hcurl-cd} with $\varepsilon = 0$ (pure advection case), and was extended to more general cases \cite{heumann2011eulerian, heumann2015stabilized}. We note that these DG methods are constructed by specifying the numerical fluxes at the interelement boundaries. 

In this paper,  we apply the weighted-residual approach \cite{brezzi2006stabilization} to derive the discontinuous Galerkin formulations for \eqref{eq:Hcurl-cd}, which establishes a linear relationship between the residual inside each element and the necessary jumps across interelement boundaries. Such an attempt was successfully applied to the scalar advection-diffusion problem \cite{ayuso2009discontinuous}, where several DG methods \cite{cockburn1999some, castillo2002optimal, zarin2005interior, sun2005symmetric} can be recovered and two new methods were developed.  

%% contribution 1: numerical scheme 
Compared to the scalar case, the Lie advection \eqref{eq:Lie-advection} becomes more complicated. To clarify the mechanism of stabilization, we first divide the vector jump into a tangential component and a normal component, by which the connection between weighted and standard averages can be established (see Lemma \ref{lm:weighted-jumps}). With the help of such results, we find that the tangential component of the weighted average corresponds to the cross jump terms, which will cancel out with a suitable choice of the test function. We note that such a cancellation effect is unique to the vector case. As a result, more flexibility can be brought in the design of stabilized DG scheme for \eqref{eq:Lie-advection}. By setting different parameters, our scheme can correspond to the generalization of a variety of classical DG schemes for the vector cases. In particular, the scheme in \cite{heumann2013stabilized} can be recovered under a special problem setting and parameter selection.

%% contribution 2: general analysis
The analysis of the scheme is established under a general assumption concerning the variable advection and reaction. Specifically, we admit the {\it degenerate Friedrichs system} which requires the minimal eigenvalue of the Friedrichs system only to be non-negative (see Assumption \ref{as:Friedrichs}). 
To the best of our knowledge, the existing works need the minimal eigenvalue to have a positive constant lower bound, in both continuous  \cite{friedrichs1958symmetric} and discrete cases \cite{ern2006discontinuous, heumann2013stabilized}. Clearly, the current analysis broadens the applicability of the proposed DG schemes, especially for the case in which the advection dominates in different parts of the domain.  

There are two crucial steps in the stability analysis. Inspired by \cite{ayuso2009discontinuous}, a weight function is introduced in Section \ref{subsec:weight} to obtain a suitable test function for the stability at the continuous level.  However, such a test function does not belong to the DG space whence a special projection has to be introduced in Section \ref{subsec:projection}. Different from the scalar case, the appearance of the jump term under the tangential direction makes the standard $L^2$  projection inadequate to control the normal component. Therefore, the special projection needs to take into account the normal moments of some facets (i.e., $(n-1)$-dimensional face), as considered in the DG schemes of scalar problems in the mixed form \cite{cockburn2008optimal,fu2015analysis}. With these two ingredients, we show the discrete inf-sup stability under a very mild mesh assumption, which is valid for the mesh size $h < h_0$ with $h_0$ independent of $\varepsilon$ (see Assumption \ref{as:normal-beta} and Appendix \ref{as:normal-beta-proof}).

The rest of the paper is organized as follows. In Section \ref{sc:setting}, we introduce the setup of the problem by presenting some relevant notations, assumptions, and DG identities. In Section \ref{sc:discretization}, we give the DG scheme in a general form. The stability analysis under a suitable DG energy norm is shown in Section \ref{sc:stability}. Section \ref{sc:error} is devoted to a priori error estimate, where the optimal convergence is proved in the DG energy norm.  Finally, in Section \ref{sc:numerical}, we present several numerical experiments to validate our theoretical results.
%% end of file

\section{Problem setting} \label{sc:setting}
In this section, we introduce some assumptions to establish the well-posedness of the DG method.  We confine the discussion to the three-dimensional case as the discretization and analysis can be easily applied to the two-dimensional translational symmetry case.  Given a bounded domain $D\subset \mathbb{R}^3$, a positive integer $s$ and $p\in[1,\infty]$, $W^{s,p}(D)$ is the Sobolev space with the corresponding usual norm and semi-norm, which are denoted by
$\|\cdot\|_{s,p,D}$ and $|\cdot|_{s.p,D}$, respectively. When $p=2$, $H^{s}(D):=W^{s,2}(D)$ with $\|\cdot\|_{s,D}=\|\cdot\|_{s,2,D}$ and $|\cdot|_{s,D}=|\cdot|_{s,2,D}$. The $L^2$-inner product on $D$ and $\partial D$ are denoted by $(\cdot, \cdot)_{D}$ and $\langle\cdot, \cdot\rangle_{\partial D}$, respectively. 

%Integration by parts, for any domain $D$
%\begin{subequations}
%\begin{align}
%(\nabla \cdot \bm{v}, w)_D &= -(\bm{v}, \nabla w)_D + \langle \bm{v} \cdot \bm{n}, w\rangle_{\partial D}, \label{eq:IBP-div} \\
%(\nabla \times \bm{v}, \bm{w})_D &= (\bm{v}, \nabla \times \bm{w})_D + \langle \bm{n} \times \bm{v}, \bm{w}\rangle_{\partial D}. \label{eq:IBP-curl}
%\end{align} 
%\end{subequations}

\subsection{Assumptions}
To begin with, motivated by \cite{devinatz1974asymptotic,
ayuso2009discontinuous}, an important assumption on the velocity field $\bm{\beta}$ is introduced.
\begin{assumption}[neither closed curve nor stationary point] \label{as:b0}
The velocity field $\bm{\beta}$ satisfies 
$$
\bm{\beta} \in \bm{W}^{1,\infty}(\Omega)  \text{ has no closed curves},
\quad \bm{\beta}(x) \neq \bm{0}, \quad \forall x \in\Omega.
$$
\end{assumption}

Under Assumption \ref{as:b0},  there exists a smooth function $\psi$ such that
\begin{equation}\label{eq:psi}
\bm{\beta} \cdot \nabla \psi(x)\ge 2b_0 \quad \forall x\in \Omega,
\end{equation}
for some constant $b_0>0$; See \cite{devinatz1974asymptotic} or
\cite[Appendix A]{ayuso2009discontinuous} for a proof.  

The next assumption gives special care to the Lie advection
$L_{{\bm{\beta}}} \bm{u}$ defined in \eqref{eq:Lie-advection}. To fix
the idea, the formal dual operator for $L_{{\bm{\beta}}}$, denoted by
$\mathcal{L}_{{\bm{\beta}}}$, is defined by
\begin{equation} \label{eq:dual-Lie}
\mathcal{L}_{{\bm{\beta}}} \bm{v} := \nabla \times ({\bm{\beta}}
  \times \bm{v}) - {\bm{\beta}} \nabla \cdot \bm{v},
\end{equation}
which yields, after a direct calculation, that for any domain $D$
\begin{subequations} \label{eq:Lie}
\begin{align}
L_{{\bm{\beta}}} \bm{u}+ \mathcal{L}_{{\bm{\beta}}}\bm{u}  &= -
  (\nabla \cdot {\bm{\beta}}) \bm{u} + \left[ \nabla {\bm{\beta}} +
  (\nabla {\bm{\beta}})^T \right] \bm{u}, \label{eq:Lie-1} \\
(L_{{\bm{\beta}}} \bm{u}, \bm{v})_D &= (\bm{u},
  \mathcal{L}_{{\bm{\beta}}}\bm{v})_D  + \langle \bm{\beta} \cdot
  \bm{n}, \bm{u} \cdot \bm{v} \rangle_{\partial D} \quad \forall
  \bm{u}, \bm{v} \in \bm{H}^{1}(D). \label{eq:Lie-2}
\end{align}
\end{subequations}
Note that no derivative of $\bm{u}$ appears on the right-hand side of
\eqref{eq:Lie-1}, which implies that half of the Lie advection in
\eqref{eq:curl-scheme-B} should be integrated by parts in the analysis
of well-posedness.  As a result, the ``effective'' reaction matrix is
assumed to be positive semi-definite, as stated below. 

\begin{assumption}[degenerate Friedrichs system]  \label{as:Friedrichs}
The smallest eigenvalue of the ``effective'' reaction matrix satisfies
\begin{equation} \label{eq:Friedrichs}
  \rho(x) := \lambda_{\min} \left[ 
(\gamma - \frac{\nabla \cdot \bm{\beta}}{2}) I + \frac{\nabla
  \bm{\beta} + (\nabla \bm{\beta})^T}{2} \right] \geq \rho_0 \geq
  0, \quad x\in \Omega.
\end{equation}
\end{assumption}

We note that the Assumption \ref{as:Friedrichs} stems from the
framework of Friedrichs symmetric operators
\cite{friedrichs1958symmetric}. The improvement in the current
work is that $\rho(x)$ does not need to have a positive lower bound,
compared to the existing DG schemes such as
\cite{ern2006discontinuous, heumann2013stabilized}.

Let $\mathcal{T}_h$ be a shape-regular family of decompositions of
$\Omega$, such that each open boundary facet belongs to either to
$\Gamma_D^\pm$, or to $\Gamma_N^\pm$ and $\mathcal{F}_h$ be the set of facets.  Let $\mathcal{F}_h^0 =
\mathcal{F}_h \setminus \partial\Omega$ be the set of interior facets
and $\mathcal{F}_h^\partial = \mathcal{F}_h \setminus \mathcal{F}_h^0$
be the set of boundary facets. The requirement of the $\mathcal{T}_h$
implies that $\mathcal{F}_h^\partial = \mathcal{F}_{h,D}^{\partial, -}
\cup \mathcal{F}_{h,D}^{\partial, +} \cup \mathcal{F}_{h,N}^{\partial,
-} \cup \mathcal{F}_{h,N}^{\partial, +}$, where 
$$
\mathcal{F}_{h,D}^{\partial, \pm} := \{F \in \mathcal{F}_h^\partial:~
F\subset \Gamma_D^\pm\}, \quad 
\mathcal{F}_{h,N}^{\partial, \pm} := \{F \in \mathcal{F}_h^\partial:~
F\subset \Gamma_N^\pm\}.
$$
We also denote by $h_T$ and $h_F$ the diameter of $T$ and $F$, respectively and set $h=\max_{T\in\mathcal{T}_h}h_T$. For $F \in \mathcal{F}_h$, we select a fixed normal unit direction, denoted by $\bm{n}_F$. Since we admit a milder Assumption \ref{as:Friedrichs} on the Friedrichs system, an additional assumption on the mesh is considered as follows. 
\begin{assumption}[normal domination] \label{as:normal-beta} 
  There exists a constant $C_{\bm{\beta}} > 0$ such that for any $T \in
\mathcal{T}_h$,
  \begin{equation} \label{eq:normal-beta}
    \|\bm{\beta}\|_{0,\infty,F} \leq C_{\bm{\beta}} \min_{x \in F}
    |\bm{\beta}(x) \cdot \bm{n}_F|
  \end{equation}
holds at least on {\it one facet} $F$ of $T$. 
\end{assumption}
\begin{remark}[verification of Assumption \ref{as:normal-beta}]
  We note that the Assumption \ref{as:normal-beta} is very mild.
  Intuitively, if $\bm{\beta}|_T$ is a constant vector,
  \eqref{eq:normal-beta} requires only one facet of $T$ to have the
  uniform positive angle between $\bm{\beta}|_T$, which is guaranteed
  by the shape regularity (in this case the constant depends only on
  the shape regularity).  In general, it can be shown that there
  exists a mesh size $h_0>0$ such that Assumption \ref{as:normal-beta}
  holds for $h\leq h_0$ (see Appendix \ref{as:normal-beta-proof} for
  details).
\end{remark}

To simplify the presentation in the analysis, we also introduce the
following set of facets 
\begin{equation} \label{eq:facet-star}
  \mathcal{F}_h^\star := \{F \in \mathcal{F}_h:~
    \|\bm{\beta}\|_{0,\infty,F} \leq C_{\bm{\beta}} \min_{x \in F}
    |\bm{\beta}(x) \cdot \bm{n}_F|\}. 
\end{equation}

\subsection{DG notation}
Let $F$ be the common facet of two elements $T^+$ and $T^-$, and
$\bm{n}^\pm $ be the unit outward norm vector on $\partial T^\pm$. We
define the jumps on the interior facet $F\in \mathcal{F}_h^0$,
$$ 
\begin{aligned}
\llbracket \bm{v} \rrbracket_{\rm t} &:= \bm{n}^+ \times \bm{v}^+ + \bm{n}^- \times \bm{v}^-, 
\qquad [\bm{v}]_{\rm n} := \bm{v}^+ \cdot \bm{n}^+ + \bm{v}^- \cdot \bm{n}^-, \\
\llbracket v \rrbracket &:= v^+\bm{n}^+ + v^-  \bm{n}^-.
\end{aligned}
$$ 
Clearly, the above jumps are independent of the orientation of the facet. For
further proposes, we also define an orientation-dependent jump $
\llbracket \bm{v} \rrbracket_F := \bm{v}^+ - \bm{v}^-$, which
satisfies 
$$ 
\llbracket \bm{v} \rrbracket_F \cdot \llbracket \bm{w} \rrbracket_F =
\llbracket \bm{v} \rrbracket_{\rm t} \cdot \llbracket \bm{w}
\rrbracket_{\rm t} + [\bm{v}]_{\rm n} [\bm{w}]_{\rm n} \quad \forall F
\in \mathcal{F}_h^0.
$$ 

With each interior facet $F$, we associate a pair of real numbers
$\alpha^\pm$, called {\it (general) weight}, satisfying $\alpha^+ + \alpha^-
= 1$.  Given a weight $\alpha$ on the interior facet, the weighted averages for
scalar-valued and vector-valued functions are respectively defined as
\begin{equation} \label{eq:weighted-avg}
\{v\}_\alpha = \alpha^+ v^+ + \alpha^- v^-, \qquad 
\vavg{\bm{v}}_\alpha = \alpha^+ \bm{v}^+ + \alpha^- \bm{v}^-.
\end{equation}
Specifically, $\alpha^\pm = 1/2$ corresponds to the standard average,
whose index is omitted for simplicity.  On the boundary facet $F \in \mathcal{F}_h^\partial$, we
set 
$$ 
\llbracket \bm{v} \rrbracket_{\rm t} := \bm{n} \times \bm{v}, \quad
[\bm{v}]_n := \bm{v} \cdot \bm{n}, \quad \llbracket v \rrbracket := v
\bm{n}, 
\quad \vavg{\bm{v}} := \bm{v}, \quad \{v\} := v.
$$ 
We also define some inner products as follows:
$$ 
(\cdot,\cdot)_{\mathcal{T}_h} = \sum_{T \in \mathcal{T}_h} (\cdot,
\cdot)_T, \qquad \langle \cdot, \cdot \rangle_{\mathcal{F}} =
\sum_{F\in \mathcal{F}}\langle \cdot, \cdot \rangle_F,
$$
where the set $\mathcal{F}$ can be taken as $\mathcal{F}_h^0$,
$\mathcal{F}_{h,D}^{\partial,\pm}$, $\mathcal{F}_{h,N}^{\partial,
\pm}$, or some of their union. The following DG identities are frequently used (cf. \cite[Eqn. (3.3)]{arnold2002unified} and \cite[Eqn. (3.9)]{perugia2003hp})
\begin{subequations}
\begin{align}
\sum_{T \in \mathcal{T}_h} \int_{\partial T} \bm{v} \cdot \bm{n} w &=
  \sum_{F\in \mathcal{F}_h} \int_F \vavg{\bm{v}} \cdot \llbracket w
  \rrbracket + \sum_{F \in \mathcal{F}_h^0} \int_F [\bm{v}]_{\rm n}
  \{w\}, \label{eq:DG-id-dot} \\
\sum_{T\in \mathcal{T}_h} \int_{\partial T}  (\bm{n} \times \bm{v})
  \cdot \bm{w} &= - \sum_{F \in \mathcal{F}_h} \int_F \vavg{\bm{v}}
  \cdot \llbracket \bm{w} \rrbracket_{\rm t} + \sum_{F \in
  \mathcal{F}_h^0} \int_{F} \llbracket \bm{v} \rrbracket_{\rm t} \cdot
  \vavg{\bm{w}}. \label{eq:DG-id-times}
\end{align}
\end{subequations}
By direct calculation, another identity that involves the velocity field $\bm{\beta}$ is given below:
\begin{equation} \label{eq:beta-id}
\llbracket \bm{\beta} \times \bm{v} \rrbracket_{\rm t} \cdot \vavg{\bm{w}} 
- [\bm{v}]_{\rm n} \{\bm{\beta} \cdot \bm{w}\}
= -\bm{\beta} \cdot \bm{n}^+ \llbracket \bm{v} \rrbracket_F \cdot \vavg{\bm{w}} \quad \text{on any }F \in \mathcal{F}_h.
\end{equation}

Next, we give some useful identities on the weighted averages. 
\begin{lemma}[weighted average identities] \label{lm:weighted-jumps}
Whenever $\bm{n}_F$ is orthogonal to $F \in \mathcal{F}_h^0$, the following relations hold:
\begin{subequations} \label{eq:weighted-jumps}
\begin{align}
\{v\}_\alpha &= \{v\} + \llbracket v \rrbracket \cdot \frac{\llbracket
  \alpha \rrbracket}{2}, \label{eq:weighted-jump}\\
\vavg{\bm{v}}_\alpha\cdot \bm{n}_F &= \left( \vavg{\bm{v}} +
  [\bm{v}]_{\rm n}  \frac{\llbracket \alpha \rrbracket}{2} \right)
  \cdot \bm{n}_F,  \label{eq:weighted-jump-n} \\
\bm{n}_F \times \vavg{\bm{v}}_\alpha &= \bm{n}_F \times \left(
  \vavg{\bm{v}} + \llbracket \bm{v}\rrbracket_{\rm t} \times
  \frac{\llbracket \alpha \rrbracket}{2} \right).
  \label{eq:weighted-jump-t}
\end{align}
\end{subequations}
\end{lemma}
\begin{proof}
We only present the proof of \eqref{eq:weighted-jump-t} as \eqref{eq:weighted-jump} is standard and 
  \eqref{eq:weighted-jump-n} has been given in \cite[Eqn.
  (2.9)]{ayuso2009discontinuous}. Using the identity $\bm{v} = (\bm{n}
  \times \bm{v}) \times \bm{n} + (\bm{v} \cdot \bm{n}) \bm{n}$, we
  have 
$$
\begin{aligned}
\bm{n}_F \times \vavg{\bm{v}}_\alpha &= \bm{n}_F \times \left(
  \vavg{\bm{v}} + (\bm{v}^+ - \bm{v}^-) \frac{(\alpha^+ -
  \alpha^-)}{2} \right) \\
&=  \bm{n}_F \times \left( \vavg{\bm{v}} + [(\bm{n}^+ \times \bm{v}^+)
  \times \bm{n}^+ -  (\bm{n}^- \times \bm{v}^-) \times \bm{n}^-]
  \frac{(\alpha^+ - \alpha^-)}{2} \right) \\
&= \bm{n}_F \times \left( \vavg{\bm{v}} + \llbracket \bm{v}
  \rrbracket_{\rm t}  \times \bm{n}^+\frac{(\alpha^+ - \alpha^-)}{2}
  \right) = \bm{n}_F \times \left( \vavg{\bm{v}} + \llbracket
  \bm{v}\rrbracket_{\rm t} \times  \frac{\llbracket \alpha
  \rrbracket}{2} \right),
\end{aligned}
$$
which proves \eqref{eq:weighted-jump-t}.
\end{proof}

We conclude this section by presenting more identities that involve the general weight $\alpha$, whose proofs are straightforward and therefore omitted. On any $F \in \mathcal{F}_h^0$, denote $[\alpha]_F := \alpha^+ - \alpha^-$, then we have 
\begin{equation} \label{eq:alpha-id1}
\begin{aligned}
& ~ [\bm{v}]_{\rm n} (\llbracket \bm{\beta} \cdot \bm{w} \rrbracket \cdot \llbracket \alpha \rrbracket) 
+ \llbracket \bm{v} \rrbracket_{\rm t} \cdot  (\llbracket \bm{\beta} \times \bm{w} \rrbracket_{\rm t} \times \llbracket \alpha \rrbracket)\\
= &~ (\bm{\beta}\cdot\llbracket\alpha\rrbracket) (\llbracket\bm{v}\rrbracket_{F} \cdot \llbracket\bm{w} \rrbracket_{F}) 
- [\alpha]_F \bm{\beta}\cdot (\llbracket\bm{v}\rrbracket_{F} [\bm{w}]_{\rm n} - \llbracket\bm{w}\rrbracket_{F} [\bm{v}]_{\rm n}),
\end{aligned}
\end{equation}
and
\begin{equation} \label{eq:alpha-id2}
\begin{aligned}
&~\llbracket \bm{\beta}\times\bm{v}\rrbracket_{\rm t}\cdot ( \llbracket \bm{w} \rrbracket_{\rm t}\times \llbracket \alpha \rrbracket) 
+ \llbracket \bm{v} \rrbracket_{\rm t} \cdot (\llbracket \bm{\beta} \times \bm{w} \rrbracket_{\rm t} \times \llbracket \alpha \rrbracket)\\
	=&~- [\alpha]_F \bm{\beta}\cdot (\llbracket\bm{v}\rrbracket_{F} [\bm{w}]_{\rm n} - \llbracket\bm{w}\rrbracket_{F} [\bm{v}]_{\rm n}).
\end{aligned}
\end{equation}
%% end of file

\section{DG scheme} \label{sc:discretization}
We rewrite the first equation in \eqref{eq:Hcurl-cd} as 
$$
\nabla \times(\varepsilon\nabla \times \bm{u} -\bm{\beta}\times
\bm{u}) + \bm{\beta} (\nabla\cdot \bm{u}) +\left[
 (\gamma-\nabla\cdot \bm{\beta})I + \nabla \bm{\beta} +(\nabla \bm{\beta})^T
 \right] \bm{u} = \bm{f}.
$$
By introducing the flux $\bm{\sigma}(\bm{u}) := \varepsilon \nabla
\times \bm{u} - \bm{\beta} \times \bm{u}$, the problem
\eqref{eq:Hcurl-cd} is equivalent to 
\begin{equation} \label{eq:Hcurl-cd-jump}
\left\{
\begin{aligned}
\nabla \times (\varepsilon \nabla \times \bm{u}) + L_{\bm{\beta}}
  \bm{u} + \gamma \bm{u} &= \bm{f} ~~\quad \text{in each } T \in
  \mathcal{T}_h, \\
\llbracket \bm{u} \rrbracket_{\rm t} &= \bm{0} ~~\quad \text{on each } F \in
  \mathcal{F}_h^0, \\
\llbracket \bm{\sigma}(\bm{u}) \rrbracket_{\rm t} &= \bm{0} ~~\quad \text{on
  each } F \in \mathcal{F}_h^0, \\
[\bm{u}]_{\rm n} &= 0 ~~\quad \text{on each } F \in \mathcal{F}_h^0, \\
\bm{n} \times \bm{u} + \chi_{\Gamma_D^-} (\bm{u} \cdot \bm{n}) \bm{n}
  &= \bm{g}_D \quad \text{on each } F \in \mathcal{F}_{h,D}^\partial, \\
\varepsilon \bm{n} \times (\nabla \times \bm{u}) + \chi_{\Gamma_N^-}
  ({\bm{\beta}} \cdot \bm{n}) \bm{u} &= \bm{g}_N \quad \text{on each }
  F \in \mathcal{F}_{h,N}^\partial.
\end{aligned}
\right.
\end{equation}

Following the residual-based approach \cite{brezzi2006stabilization,ayuso2009discontinuous}, we shall introduce a variational formulation of \eqref{eq:Hcurl-cd-jump} in which each of the equations above has the same relevance and is therefore treated in the same fashion. To this end, we introduce the space
$$
\bm{V}(\mathcal{T}_h) := \{\bm{v} \in \bm{L}^2(\Omega)~:~\bm{v}|_{T}
\in \bm{H}^s(T), \forall T \in \mathcal{T}_h, ~s > 3/2\},
$$ 
and we assume the operators $\mathcal{B}_1$, $\mathcal{B}_2$,
$\mathcal{B}_3$, $\mathcal{B}_1^D$, $\mathcal{B}_2^N$ mapping from
$\bm{V}(\mathcal{T}_h)$ to $\bm{L}^2(\mathcal{F}_h^0)$,
$\bm{L}^2(\mathcal{F}_h^0)$, $L^2(\mathcal{F}_h^0)$,
$\bm{L}^2(\Gamma_D)$, $\bm{L}^2(\Gamma_N)$, respectively. Then, we
consider the problem: Find $\bm{u} \in \bm{V}(\mathcal{T}_h)$ such
that for any $ \bm{v} \in \bm{V}(\mathcal{T}_h)$,
\begin{equation} \label{eq:curl-scheme-B}
\begin{aligned}
&\int_\Omega \left[\nabla_h \times(\varepsilon \nabla_h \times \bm{u})
  + L_{{\bm{\beta}},h} \bm{u} + \gamma \bm{u} - \bm{f}\right] \cdot
  \bm{v} \\
+& \sum_{F \in \mathcal{F}_h^0} \int_F \llbracket \bm{u}
  \rrbracket_{\rm t} \cdot \mathcal{B}_1 \bm{v} + \llbracket
  \bm{\sigma}(\bm{u}) \rrbracket_{\rm t} \cdot \mathcal{B}_2\bm{v} +
  [\bm{u}]_{\rm n} \mathcal{B}_3 \bm{v} \\
+ & \sum_{F \in \mathcal{F}_{h,D}^\partial} \int_F \left[ \bm{n} \times \bm{u} +
  \chi_{\Gamma_D^-} (\bm{u} \cdot \bm{n}) \bm{n} - \bm{g}_D\right]
  \cdot \mathcal{B}_1^D \bm{v} \\
+ & \sum_{F\in \mathcal{F}_{h,N}^\partial} \int_F\left[\varepsilon \bm{n} \times
  (\nabla_h \times \bm{u}) + \chi_{\Gamma_N^-} ({\bm{\beta}} \cdot
  \bm{n}) \bm{u} - \bm{g}_N\right] \cdot \mathcal{B}_2^N \bm{v} = 0.
\end{aligned}
\end{equation}
Here, $\nabla_h \times$ and $L_{\bm{\beta},h}$ denote the curl and Lie convection operators element by element, respectively. As a distinctive feature, the solution of the original problem \eqref{eq:Hcurl-cd} coincides with the solution of \eqref{eq:curl-scheme-B} with an arbitrary choice of $\mathcal{B}$'s operators.  

In this paper, we consider a general choice as follows:
\begin{subequations} \label{eq:general-Bs}
\begin{align}
\mathcal{B}_1 \bm{v} &= \underbrace{ \varepsilon\eta_F h_F^{-1}
  \llbracket \bm{v} \rrbracket_{\rm t} - \theta \vavg{  \varepsilon
  \nabla_h \times \bm{v} }_{\alpha_d}}_{\mathcal{B}_{11} \bm{v}}
\underbrace{-\vavg{ {\bm{\beta}} \times \bm{v}}_{\alpha_d} + \vavg{
  {\bm{\beta}} \times \bm{v}}_{\alpha}}_{\mathcal{B}_{12}\bm{v}}, \\
\mathcal{B}_2 \bm{v} &= -\vavg{\bm{v}}_{1- \alpha_d}, \\
\mathcal{B}_3 \bm{v} &= -\{{\bm{\beta}} \cdot \bm{v}\}_{1-\alpha} +
  \tau_F |\llbracket \alpha - \alpha_d \rrbracket| [\bm{v}]_n, \\
\mathcal{B}_1^D\bm{v} &= \underbrace{\varepsilon \eta_{F} h_F^{-1}
  \bm{n} \times \bm{v} - \theta [\bm{n} \times (\varepsilon \nabla_h
  \times \bm{v})] \times \bm{n} }_{\mathcal{B}_{11}^D\bm{v}}  \\
& \quad \underbrace{- ({\bm{\beta}} \cdot \bm{n})[\bm{n} \times \bm{v}
  + (\bm{n} \cdot \bm{v}) \bm{n}] 
  \chi_{\Gamma_D^-}}_{\mathcal{B}_{12}^D \bm{v}} ,  \notag \\
\mathcal{B}_2^N\bm{v} &= -\bm{v}.
\end{align}
\end{subequations}
Here, $\alpha$ and $\alpha_d$ are the general weights that are used to define the weighted averages \eqref{eq:weighted-avg}, $\theta \in [-1,1]$, $\tau_F$ and $\eta_F$ are the positive parameters that will be discussed later. Next, we derive the bilinear form for this choice.  

\paragraph{Part 1: diffusion} In view of $\mathcal{B}_2 \bm{v}$ and
\eqref{eq:weighted-jump-t}, we have 
$$ 
\begin{aligned}
&\quad \sum_{F\in \mathcal{F}_h^0} \int_F \llbracket
  \bm{\sigma}(\bm{u}) \rrbracket_{\rm t} \cdot \mathcal{B}_2 \bm{v} 
= \langle \llbracket \bm{\sigma}(\bm{u}) \rrbracket_{\rm t},
  -\vavg{\bm{v}}_{1 - \alpha_d} \rangle_{\mathcal{F}_h^0} \\
&= -\langle \llbracket \bm{\sigma}(\bm{u}) \rrbracket_{\rm t},
  \vavg{\bm{v}} \rangle_{\mathcal{F}_h^0} - 
\langle \llbracket \bm{v} \rrbracket_{\rm t},  \llbracket \varepsilon
  \nabla_h \times \bm{u} \rrbracket_{\rm t} \times \frac{\llbracket
  \alpha_d
\rrbracket}{2} \rangle_{\mathcal{F}_h^0} - \langle \llbracket
  {\bm{\beta}} \times \bm{u} \rrbracket_{\rm t}, \llbracket \bm{v}
  \rrbracket_{\rm t} \times \frac{\llbracket \alpha_d \rrbracket}{2}
  \rangle_{\mathcal{F}_h^0}.
\end{aligned}
$$
Further, on the Dirichlet boundary, we have 
$$ 
\sum_{F\in \mathcal{F}_{h,D}^\partial} \int_F (\bm{n} \times \bm{u}) \cdot
\mathcal{B}_{11}^D \bm{v} =
 \langle \varepsilon \eta_F h_F^{-1} \llbracket \bm{u} \rrbracket_t,
 \llbracket \bm{v} \rrbracket_t \rangle_{\mathcal{F}_{h,D}^\partial} -
 \theta \langle 
 \llbracket \bm{u} \rrbracket_{\rm t}, \vavg{\varepsilon \nabla_h
 \times \bm{v}} \rangle_{\mathcal{F}_{h,D}^\partial}.
$$  
Using above equations, together with \eqref{eq:DG-id-times}, we have 
\begin{equation} \label{eq:1-diffusion}
\begin{aligned}
& ~(\nabla_h \times (\varepsilon \nabla_h \times \bm{u}), \bm{v}) +
  \sum_{F \in \mathcal{F}_h^0} \int_F \llbracket \bm{u}
  \rrbracket_{\rm t} \cdot \mathcal{B}_{11} \bm{v} + \llbracket
  \bm{\sigma}(\bm{u}) \rrbracket_{\rm t} \cdot \mathcal{B}_2\bm{v} \\
& ~ + \sum_{F\in \mathcal{F}_{h,D}^\partial} \int_F (\bm{n} \times \bm{u}) \cdot
  \mathcal{B}_{11}^D \bm{v} + \sum_{F\in \mathcal{F}_{h,N}^\partial}
  \int_F\left[\varepsilon \bm{n} \times (\nabla_h \times
  \bm{u})\right] \cdot \mathcal{B}_2^N \bm{v}\\
%= &~ (\varepsilon \nabla_h \times \bm{u}, \nabla_h \times \bm{v}) -
%  \langle \vavg{\varepsilon \nabla_h \times \bm{u}}, \llbracket \bm{v}
%  \rrbracket_{\rm t}\rangle_{\mathcal{F}_h \cup
%  \mathcal{F}_{h,D}^\partial } + \langle \llbracket \varepsilon
%  \nabla_h \times \bm{u} \rrbracket_{\rm t}, \vavg{\bm{v}}
%  \rangle_{\mathcal{F}_h^0} \\
%& ~ + \sum_{F \in \mathcal{F}_h^0} \int_F \llbracket \bm{u}
%  \rrbracket_{\rm t} \cdot \mathcal{B}_{11} \bm{v} + \llbracket
%  \bm{\sigma}(\bm{u}) \rrbracket_{\rm t} \cdot \mathcal{B}_2\bm{v} +
%  \sum_{F\in \Gamma_D} \int_F (\bm{n} \times \bm{u}) \cdot
%  \mathcal{B}_{11}^D \bm{v} \\
= & ~ (\varepsilon \nabla_h \times \bm{u}, \nabla_h \times \bm{v}) -
  \langle \vavg{\varepsilon \nabla_h \times \bm{u}}_{\alpha_d},
  \llbracket \bm{v} \rrbracket_{\rm t} \rangle_{\mathcal{F}_h^0 \cup
  \mathcal{F}_{h,D}^\partial} + \sum_{F \in \mathcal{F}_h^0} \int_F
  \llbracket \bm{u} \rrbracket_{\rm t} \cdot \mathcal{B}_{11} \bm{v}
  \\
& + \sum_{F\in \mathcal{F}_{h,D}^\partial} \int_F (\bm{n} \times \bm{u}) \cdot
  \mathcal{B}_{11}^D \bm{v} +\langle \llbracket {\bm{\beta}} \times
  \bm{u} \rrbracket_{\rm t}, \vavg{\bm{v}} \rangle_{\mathcal{F}_h^0} -
\langle \llbracket {\bm{\beta}} \times \bm{u} \rrbracket_{\rm t},
  \llbracket \bm{v} \rrbracket_{\rm t} \times \frac{\llbracket
  \alpha_d \rrbracket}{2} \rangle_{\mathcal{F}_h^0}  \\
= & ~ (\varepsilon \nabla_h \times \bm{u}, \nabla_h \times \bm{v}) -
  \langle \vavg{\varepsilon \nabla_h \times \bm{u}}_{\alpha_d},
  \llbracket \bm{v} \rrbracket_{\rm t} \rangle_{\mathcal{F}_h^0 \cup
  \mathcal{F}_{h,D}^\partial} \\
&~ - \theta \langle \llbracket \bm{u} \rrbracket_{\rm t},
  \vavg{\varepsilon \nabla_h \times \bm{v} }_{\alpha_d}
  \rangle_{\mathcal{F}_h^0 \cup \mathcal{F}_{h,D}^\partial} + \langle
  \varepsilon \eta_F h_F^{-1}  \llbracket \bm{u} \rrbracket_{\rm t},
  \llbracket \bm{v} \rrbracket_{\rm t}  \rangle_{\mathcal{F}_h^0 \cup
  \mathcal{F}_{h,D}^\partial} \\
&~ +\langle \llbracket {\bm{\beta}} \times \bm{u} \rrbracket_{\rm t},
  \vavg{\bm{v}}_{1-\alpha_d} \rangle_{\mathcal{F}_h^0}.
\end{aligned}
\end{equation}
Here, we see that $\theta \in [-1,1]$ is a parameter that allows up to include
various formulations for treating the diffusive part: symmetric for
$\theta = 1$, skew-symmetric for $\theta = -1$, and neutral for
$\theta = 0$. 

\paragraph{Part 2: reaction-convection} The remaining terms are the interior part
\begin{equation}
\sum_{F\in \mathcal{F}_h^0} \int_F \llbracket \bm{u} \rrbracket_{\rm
  t} \cdot \mathcal{B}_{12} \bm{v} + \sum_{F\in \mathcal{F}_h^0}
  \int_F [\bm{u}]_{\rm n} \mathcal{B}_3\bm{v},
\end{equation}
and the boundary part
\begin{equation} \label{eq:2-Lie-boundary}
\begin{aligned}
\sum_{F \in \mathcal{F}_{h,N}^\partial} \int_F \chi_{\Gamma_N^-} (\bm{\beta} \cdot
  \bm{n}) \bm{u} \cdot \mathcal{B}_2^N\bm{v} &= - \langle \bm{\beta}
  \cdot \bm{n}, \bm{u} \cdot \bm{v}
  \rangle_{\mathcal{F}_{h,N}^{\partial, -}}, \\
\sum_{F \in \mathcal{F}_{h,D}^\partial} \int_F (\bm{n} \times \bm{u}) \cdot
  \mathcal{B}_{12}^D\bm{v} + \chi_{\Gamma_D^-} (\bm{u} \cdot \bm{n})
  \bm{n} \cdot \mathcal{B}^D\bm{v} &= - \langle \bm{\beta} \cdot
  \bm{n}, \bm{u} \cdot \bm{v} \rangle_{\mathcal{F}_{h,D}^{\partial,
  -}}.
\end{aligned}
\end{equation}

Combining \eqref{eq:1-diffusion}--\eqref{eq:2-Lie-boundary}, we obtain
the following scheme: Find $\bm{u} \in \bm{V}(\mathcal{T}_h)$ such
that 
\begin{equation} \label{eq:variational-form}
a_h(\bm{u}, \bm{v}) = F(\bm{v}),\quad \forall \bm{v}\in\bm{V}(\mathcal{T}_h),
\end{equation}
where $a_h(\bm{u}, \bm{v}) := a_h^d(\bm{u}, \bm{v}) + a_h^{rc}(\bm{u},
\bm{v})$ with 
\begin{subequations}\label{eq:partvariation-form}
\begin{align}
a_h^d(\bm{u}, \bm{v}) &:=  (\varepsilon \nabla_h \times \bm{u},
  \nabla_h \times \bm{v}) 
- \langle \vavg{\varepsilon \nabla_h \times \bm{u}}_{\alpha_d},
  \llbracket \bm{v} \rrbracket_{\rm t} \rangle_{\mathcal{F}_h^0 \cup
  \mathcal{F}_{h,D}^\partial} \label{eq:ad} \\
&~\quad - \theta \langle \llbracket \bm{u} \rrbracket_{\rm t},
  \vavg{\varepsilon \nabla_h \times \bm{v} }_{\alpha_d}
  \rangle_{\mathcal{F}_h^0 \cup \mathcal{F}_{h,D}^\partial} + \langle
  \varepsilon \eta_F h_F^{-1}  \llbracket \bm{u} \rrbracket_{\rm t},
  \llbracket \bm{v} \rrbracket_{\rm t}  \rangle_{\mathcal{F}_h^0 \cup
  \mathcal{F}_{h,D}^\partial}, \notag \\
a_h^{rc}(\bm{u}, \bm{v}) &:= (L_{\bm{\beta},h}\bm{u}+\gamma
  \bm{u},\bm{v}) +\langle \llbracket \bm{\beta}\times
  \bm{u}\rrbracket_{\rm
  t},\vavg{\bm{v}}_{1-\alpha_d}\rangle_{\mathcal{F}_h^0}-\langle\left[\bm{u}\right]_{\rm{n}},
    \{\bm{\beta}\cdot\bm{v}\}_{1-\alpha}\rangle_{\mathcal{F}_h^0}
    \notag \\
    &~\quad +\langle\llbracket \bm{u}\rrbracket_{\rm
    t},\vavg{\bm{\beta}\times \bm{v}}_\alpha-\vavg{\bm{\beta}\times
    \bm{v}}_{\alpha_d}\rangle_{\mathcal{F}_h^0}+\langle
    \tau_F|\llbracket\alpha-\alpha_d\rrbracket|[\bm{u}]_{\rm{n}},[\bm{v}]_{\rm{n}}\rangle_{\mathcal{F}_h^0}
    \label{eq:arc} \\
    &~\quad -\langle
    \bm{\beta}\cdot\bm{n},\bm{u}\cdot\bm{v}\rangle_{\mathcal{F}_{h,D}^{\partial,-}
    \cup \mathcal{F}_{h,N}^{\partial,-} }, \notag \\
F(\bm{v}) &:= (\bm{f}, \bm{v}) + \langle \bm{g}_D, \varepsilon \eta_F
  h_F^{-1} \bm{n} \times \bm{v} - \theta [\bm{n} \times (\varepsilon
  \nabla_h \times \bm{v})] \times \bm{n}
  \rangle_{\mathcal{F}_{h,D}^\partial} \\ 
&~\quad -\langle (\bm{\beta} \cdot \bm{n}) \bm{g}_D, \bm{n} \times
  \bm{v} + (\bm{n} \cdot \bm{v}) \bm{n}
  \rangle_{\mathcal{F}_{h,D}^{\partial, -}} - \langle \bm{g}_N,
  \bm{v}\rangle_{\mathcal{F}_{h,N}^\partial} \notag.
\end{align}
\end{subequations}

We will show in Section \ref{sc:stability} that a proper choice of $\alpha$ corresponds to the upwind stabilization, and $\alpha_d$ corresponds to the cancellation of the cross jump terms. Two parameters $\eta_F$ and $\tau_F$ play different roles: the former penalizes the tangential jump which is assumed to be large enough, and the latter stabilizes the normal jump which is assumed only to be positive. Further, we denote
\begin{equation} \label{eq:eta-tau}
\eta_0 := \min_{F \in \mathcal{F}_h^0 \cup \mathcal{F}_{h,D}^\partial}
\eta_F, \qquad \tau_0 := \min_{F \in \mathcal{F}_h^0} \tau_F. 
\end{equation}

\begin{remark}[new mechanism in vector case]
Note that in the pure advection case ($\varepsilon=0$) with a special choice $\alpha=\alpha_d$, subtracting \eqref{eq:alpha-id1} from \eqref{eq:alpha-id2}, and using \eqref{eq:Lie-2}, \eqref{eq:beta-id}, we obtain
 $$
 \begin{aligned}
a_h^{rc}(\bm{u},\bm{v}) 
& = (L_{\bm{\beta},h}\bm{u}+\gamma
 \bm{u},\bm{v}) +\langle \llbracket \bm{\beta}\times
 \bm{u}\rrbracket_{\rm
 	t},\vavg{\bm{v}}_{1-\alpha}\rangle_{\mathcal{F}_h^0}-\langle\left[\bm{u}\right]_{\rm{n}},
 \{\bm{\beta}\cdot\bm{v}\}_{1-\alpha}\rangle_{\mathcal{F}_h^0}\\&\quad -\langle
 \bm{\beta}\cdot\bm{n},\bm{u}\cdot\bm{v} \rangle_{\mathcal{F}_{h,D}^{\partial,-}
    \cup \mathcal{F}_{h,N}^{\partial,-} }\\
& =(\bm{u},\mathcal{L}_{\bm{\beta},h}\bm{v}+\gamma \bm{v}) +\langle \bm{\beta}\cdot \bm{n}^+\vavg{\bm{u}},\llbracket\bm{v}\rrbracket_{F}\rangle_{\mathcal{F}_h^0} \\
& \quad + \langle \frac{\alpha^+-\alpha^-}{2}\bm{\beta}\cdot \bm{n}^+\llbracket\bm{u}\rrbracket_{F},\llbracket\bm{v}\rrbracket_{F}\rangle_{\mathcal{F}_h^0}
   +\langle\bm{\beta}\cdot \bm{n},\bm{u}\cdot\bm{v}\rangle_{\mathcal{F}_{h}^{\partial,+}},
 \end{aligned}
 $$
 which recovers the scheme introduced in \cite[Eqn. (2.8)]{heumann2013stabilized} with $c_F=\frac{\alpha^+-\alpha^-}{2}$ if $\bm{n}_F=\bm{n}^+$ for facet $F$. In (\ref{eq:partvariation-form}), the roles of $\alpha$ and $\alpha_d$ are separated due to a new mechanism compared to the scalar case: the following cross jump terms
 $$
 \llbracket \bm{u} \rrbracket_{\rm t} \cdot \left( \llbracket \bm{\beta} \times \bm{v} \rrbracket_{\rm t} \times \frac{\llbracket {\alpha_d} \rrbracket}{2} \right) 
 +\llbracket \bm{\beta} \times \bm{u} \rrbracket_{\rm t} \cdot \left( \llbracket\bm{v} \rrbracket_{\rm t} \times \frac{\llbracket {\alpha_d} \rrbracket}{2} \right),
 $$
 will not vanish in general unless $\bm{v}=\bm{u}\zeta$ for some scalar smooth function $\zeta$. 
\end{remark}
%% end of file

\section{Stability analysis} \label{sc:stability}
With any integer $k\ge 1$, we use the finite element space of
discontinuous polynomial functions
$$
\bm{V}_h^k:=\{\bm{v}\in {{\bm{L}}}^2(\Omega):\bm{v}_{\mid T}\in
\bm{\mathcal{P}}_k(T) \quad\forall T\in\mathcal{T}_h\},
$$
where ${{\bm{\mathcal{P}}}}_k(T)$ is the space of polynomials of
degree at most $k$ on $T$. We have the following discrete problem:
\begin{equation} \label{eq:variational-form2}
	\left\{
	\begin{aligned}
		&\text{Find }\bm{u}_h\in\bm{V}_h^k \text{ such that }\\
		& a_h(\bm{u}_h,\bm{v}_h)=F(\bm{v}_h), \quad \forall \bm{v}_h\in \bm{V}_h^k.
	\end{aligned}
	\right.
\end{equation}
The variational form is consistent by construction, so that
\begin{equation}\label{eq:consistency}
	a_h(\bm{u}-\bm{u}_h,\bm{v}_h)=0 \quad \forall \bm{v}_h\in \bm{V}_h^k.
\end{equation}
where $\bm{u}$ is the solution of (\ref{eq:Hcurl-cd}).

We will prove the stability in the following norm
\begin{equation}\label{eq:energynorm}
\triplenorm{\bm{v}}^2:=\triplenorm{\bm{v}}^2_{d}+\triplenorm{\bm{v}}^2_{rc},
\end{equation}
with
$$
\begin{aligned}
&\triplenorm{\bm{v}}^2_{d}:=\varepsilon\|\nabla_h\times\bm{v}\|_{0,\Omega}^2+\varepsilon
  \|\bm{v}\|_{\rm
  t}^2:=\varepsilon\|\nabla_h\times\bm{v}\|_{0,\Omega}^2+\varepsilon\sum_{F\notin
  \Gamma_N}h_F^{-1}\|\llbracket\bm{v}\rrbracket_{\rm t}\|_{0,F}^2,\\
  &\triplenorm{\bm{v}}^2_{rc}:=\|(\bar{\rho}+b_0)^{1/2}\bm{v}\|_{0,\Omega}^2+\sum_{F\in
  \mathcal{F}_h}\||\bm{\beta}\cdot\bm{n}|^{1/2}\llbracket\bm{v}\rrbracket_{F}||^2_{0,F}
  + \|\bm{v}\|_{\rm{n}}^2, \\
  &\|\bm{v}\|_{\rm n}^2:=\sum_{F \in
  \mathcal{F}_h^0}|\llbracket\alpha-\alpha_d\rrbracket|
  \|[\bm{v}]_{\rm n}\|^2_{0,F},
	\end{aligned}
$$
where $\bar{\rho}(x)$ is a piecewise constant function defined as
\begin{equation}
  \bar{\rho}(x)_{\mid T} :=\bar{\rho}_{\mid T} :=\min _{x \in T}
  \rho(x) \quad \forall T \in \mathcal{T}_{h}.
\end{equation}

Using the local trace inequality and inverse inequality, we get for
any $\bm{w} \in \bm{V}_h^k$ and $\bm{v} \in \bm{V}(\mathcal{T}_h)$,
\begin{equation}\label{eq:traceinverse}
\begin{aligned}
&~~ \langle \vavg{\varepsilon \nabla_h \times
  \bm{w}}_{\alpha_d},\llbracket\bm{v}\rrbracket_{\rm
  t}\rangle_{\mathcal{F}_h^0 \cup \mathcal{F}_{h,D}^\partial} \\
%&\le\sum_{F\notin \Gamma_N} \varepsilon h_F \| \vavg{\nabla_h\times \bm{w}}_{\alpha_d}\|_{0,F}h_F^{-1}\|\llbracket\bm{v}\rrbracket_{\rm t}\|_{0,F}\\
\le&~ \varepsilon (\sum_{F\notin \Gamma_N}h_F \| \vavg{\nabla_h\times
  \bm{w}}_{\alpha_d}\|_{0,F}^2)^{1/2}(\sum_{F\notin
  \Gamma_N}h_F^{-1}\|\llbracket\bm{v}\rrbracket_{\rm
  t}\|_{0,F}^2)^{1/2}\\
%\le&~ C_g\alpha_d^{\max}\varepsilon (\sum_{F\notin
%  \Gamma_N}\sum_{\substack{T\in\mathcal{T}_h\\F\in T}}
%  \|\nabla_h\times\bm{w}\|^2_{0,T} +
%  h_F^2|\nabla_h\times\bm{w}|_{1,T}^2)^{1/2}\|\bm{v}\|_{\rm t}\\
\le&~ C_g\alpha_d^{\max}\varepsilon
  \|\nabla_h\times\bm{w}\|_{0,\Omega}\|\bm{v}\|_{\rm t},
\end{aligned}
\end{equation}
where $C_g$ depends only on the degree of the polynomials and shape-regularity constant, and $\alpha_d^{\max}$ is defined
by $\alpha_d^{\max} :=
\sup_{F\in\mathcal{F}_h^0}(|\alpha_d^+|+|\alpha_d^-|)$.  Next, we will
prove the stability in the norm (\ref{eq:energynorm}) through an inf-sup
condition. To this end, we shall discuss two ingredients for generalizing the condition of the Friedrichs system: a weight function in Section \ref{subsec:weight} and a tailored projection in Section \ref{subsec:projection}.

\subsection{Weight function} \label{subsec:weight}
Using the smooth function $\psi$ defined in Assumption \ref{as:b0} (neither closed curve nor stationary point), we
introduce the weight function \cite{ayuso2009discontinuous}
\begin{equation}\label{eq:weightfunction}
	\varphi := e^{-\psi}+\kappa := \chi + \kappa,
\end{equation}
where $\kappa$ is a positive constant to be determined later.
Assumption \ref{as:b0} demonstrates that there exist positive
constants $\chi^*,\chi_*,\tilde{\chi}$ such that
\begin{equation}\label{eq:chiinequality}
	\chi_*\le\chi\le\chi^*,\quad |\nabla\chi|\le \tilde{\chi}.
\end{equation}

 The following lemma is a generalization of results in \cite[Lemma
 4.1]{ayuso2009discontinuous} for magnetic advection-diffusion problems. 
\begin{lemma}[weighted coercivity] \label{lm:weightinfsup}
  Let $a_h(\cdot,\cdot)$ be defined in (\ref{eq:variational-form}),
  with $\tau_0 \geq 1$, $\eta_0$ and $\kappa$ satisfying
\begin{subequations}\label{eq:chicondition}
	\begin{align}
  \eta_0 & \geq \max\{8C_g^2 (\theta+1)^2(\alpha_d^{\max})^2, 1\},
    \label{eq:chicondition1}\\ 
  \chi_*+\kappa& \geq \max\{\frac{\chi^*+\kappa}{2},
    \frac{2\tilde{\chi}^2}{b_0\chi_*}\}, \label{eq:chicondition2}
	\end{align}
\end{subequations}
and the corresponding $\varphi$ is defined in
(\ref{eq:weightfunction}). If $\alpha$ is taken to be an upwind
direction, i.e. $\bm{\beta}\cdot\llbracket\alpha\rrbracket \geq
C_{up} |\bm{\beta} \cdot \bm{n}|$ for all $F\in \mathcal{F}_h^0$ ,
then the following inequalities hold for all $\bm{v_h}\in \bm{V}_h^k$,
\begin{subequations}
\begin{align}
a_h(\bm{v}_h,\bm{v}_h\varphi) &\ge \frac{\chi_*+\kappa}{2}
  \triplenorm{\bm{v}_h}_{d}^2 +
  \frac{\chi_*\min\{C_{up},1\}}{2}\triplenorm{\bm{v}_h}^2_{rc}
  \label{eq:weightinfsup1} \\
& \ge \frac{\chi_* \min\{C_{up}, 1\}}{2}\triplenorm{\bm{v}_h}^2,
  \notag \\ \triplenorm{\bm{v}_h\varphi} &\le
  \sqrt{2}(\chi^*+\kappa)\triplenorm{\bm{v}_h}.
  \label{eq:weightinfsup2}
\end{align}
\end{subequations}
\end{lemma}
\begin{proof}
Using (\ref{eq:traceinverse}) for the diffusion part, we have
$$
\begin{aligned}
& ~~ a_h^d(\bm{v}_h, \bm{v}_h\varphi)\\
= &~ (\varepsilon\nabla_h\times\bm{v}_h,\varphi \nabla_h\times
  \bm{v}_h ) +\langle
  \varepsilon\eta_Fh_F^{-1}\llbracket\bm{v}_h\rrbracket_{\rm
  t},\llbracket\bm{v}_h\varphi\rrbracket_{\rm
  t}\rangle_{\mathcal{F}_h^0 \cup \mathcal{F}_{h,D}^\partial}\\
  &~ - (\theta+1)\langle
  \vavg{\varepsilon\nabla_h\times\bm{v}_h}_{\alpha_d},\llbracket\bm{v}_h\rrbracket_{\rm
  t}\varphi\rangle_{\mathcal{F}_h^0 \cup \mathcal{F}_{h,D}^\partial} 
   + (\varepsilon\nabla_h\times\bm{v}_h,\nabla\varphi\times\bm{v_h})\\
\ge &~ \varepsilon\Big[
  (\chi_*+\kappa)(\|\nabla_h\times\bm{v}_h\|_{0,\Omega}^2+\eta_0\|\bm{v}_h\|_{\rm
  t}^2)\\
  &~-C_g(\chi^*+\kappa)\alpha_d^{\max}|\theta+1| \cdot
  \|\nabla_h\times\bm{v}_h\|_{0,\Omega}\|\bm{v}_h\|_{\rm
  t}-\tilde{\chi}\|\nabla_h\times\bm{v}_h\|_{0,\Omega}\|\bm{v}_h\|_{0,\Omega}\Big].
\end{aligned}
$$
Further, \eqref{eq:chicondition1} and \eqref{eq:chicondition2} imply
  that 
$$
\begin{aligned}
 C_g(\chi^*+\kappa)\alpha_d^{\max}|(\theta+1)|& \cdot 
  \|\nabla_h\times\bm{v}_h\|_{0,\Omega}\|\bm{v}_h\|_{\rm t} \\
%&\le
%\frac{\chi_*+\kappa}{4} \|\nabla_h\times \bm{v}_h\|_{0,\Omega}^2 +
%  4C_g^2(\chi_*+\kappa)
%  (\alpha_d^{\max})^2(\theta+1)^2\|\bm{v}_h\|_{\rm t}^2, \\
&\le \frac{\chi_*+\kappa}{4} \|\nabla_h\times \bm{v}_h\|_{0,\Omega}^2 +
  \frac{\chi_*+\kappa}{2} \eta_0 \|\bm{v}_h\|_{\rm t}^2, \\
  \tilde{\chi}\|\nabla_h\times\bm{v}_h\|_{0,\Omega}\|\bm{v}_h\|_{0,\Omega} & \le  \frac{\chi_* +
  \kappa}{4}\|\nabla_h\times\bm{v}_h\|_{0,\Omega}^2 +
  \frac{b_0\chi_*}{2} \|\bm{v}_h\|_{0,\Omega}^2,
\end{aligned}
$$
%and \eqref{eq:chicondition2} implies
%\begin{equation}\label{eq:curll2term}
%\begin{aligned}
%	\tilde{\chi}\|\nabla_h\times\bm{v}_h\|_{0,\Omega}\|\bm{v}_h\|_{0,\Omega}
%  \le&~
%  \frac{\tilde{\chi}^2}{2b_0\chi_*}\|\nabla_h\times\bm{v}_h\|_{0,\Omega}^2
%  + \frac{b_0\chi_*}{2} \|\bm{v}_h\|_{0,\Omega}^2 \\
%  \le&~  \frac{\chi_* +
%  \kappa}{4}\|\nabla_h\times\bm{v}_h\|_{0,\Omega}^2 +
%  \frac{b_0\chi_*}{2} \|\bm{v}_h\|_{0,\Omega}^2,
%\end{aligned}
%\end{equation}
whence 
\begin{equation}\label{eq:infsupdiffusion}
\begin{aligned}
a_h^d(\bm{v}_h,&\bm{v}_h\varphi)\ge \varepsilon \left[
  \frac{\chi_*+\kappa}{2}(\|\nabla_h\times\bm{v}_h\|_{0,\Omega}^2 +
  \|\bm{v}_h\|_{\rm t}^2) - \frac{b_0 \chi_*}{2}
  \|\bm{v}_h\|_{0,\Omega}^2 \right].
\end{aligned}
\end{equation}

As for the reactive-convective part, we first notice that the dual operator of the Lie derivative in \eqref{eq:dual-Lie} satisfies 
$$
\begin{aligned}
  \mathcal{L}_{\bm{\beta},h}(\bm{v}_h\varphi) &= \varphi
  \mathcal{L}_{\bm{\beta},h}\bm{v}_h+
  \nabla\varphi\times(\bm{\beta}\times\bm{v}_h)-(
  \nabla\varphi\cdot\bm{v}_h)\bm{\beta} = \varphi
  \mathcal{L}_{\bm{\beta},h}\bm{v}_h-(\bm{\beta}\cdot\nabla\varphi)\bm{v}_h.
\end{aligned}
$$
Then, using the properties of Lie derivative \eqref{eq:Lie}, we get
\begin{equation}\label{eq:infsupLbeta}
\begin{aligned}
&\quad (L_{\bm{\beta},h}\bm{v}_h+\gamma \bm{v}_h,\bm{v}_h\varphi) -
  \langle \bm{\beta}\cdot\bm{n},\varphi \bm{v}_h\cdot\bm{v}_h
  \rangle_{\mathcal{F}_{h,D}^{\partial,-} \cup
  \mathcal{F}_{h,N}^{\partial,-} }\\
%=& ~ (\frac{1}{2}L_{\bm{\beta},h}\bm{v}_h+\gamma
%  \bm{v}_h,\bm{v}_h\varphi) +
%  (\bm{v}_h,\frac{1}{2}\mathcal{L}_{\bm{\beta},h}(\bm{v}_h\varphi)) \\
%&~ ~+\frac{1}{2} \sum_{T \in \mathcal{T}_h} \langle \bm{\beta} \cdot
%  \bm{n},  \varphi \bm{v}_h \cdot \bm{v}_h \rangle_{\partial T}
%- \langle \bm{\beta}\cdot\bm{n},\varphi |\bm{v}_h|^2
%  \rangle_{\mathcal{F}_{h,D}^{\partial,-} \cup
%  \mathcal{F}_{h,N}^{\partial,-} }  \\
=&~(\frac{1}{2}L_{\bm{\beta},h}\bm{v}_h+\frac{1}{2}\mathcal{L}_{\bm{\beta},h}\bm{v}_h+\gamma
  \bm{v}_h,\bm{v_h}\varphi) 
	-\frac{1}{2}(\bm{\beta}\cdot \nabla \varphi,\bm{v}_h\cdot \bm{v}_h)\\
  &~ ~+\frac{1}{2}\langle \bm{\beta} \varphi,\llbracket \bm{v}_h \cdot
  \bm{v}_h \rrbracket\rangle_{\mathcal{F}_h^0} 
+ \frac12\langle |\bm{\beta} \cdot \bm{n}|\varphi, |\bm{v}_h|^2
  \rangle_{\mathcal{F}_h^\partial}  \\
  \ge&~
  \chi_*\|(\bar{\rho}+b_0)^{1/2}\bm{v}_h\|_{0,\Omega}^2+\frac{1}{2}\langle
  \bm{\beta}\varphi,\llbracket\bm{v_h}\cdot\bm{v_h}\rrbracket\rangle_{\mathcal{F}_h^0}+\frac{1}{2}\langle
  |\bm{\beta}\cdot
  \bm{n}|\varphi,|\bm{v}_h|^2\rangle_{\mathcal{F}_{h}^{\partial} }.
	\end{aligned}
\end{equation}
Here, we apply \eqref{eq:psi} in the last step to deduce
$-\bm{\beta}\cdot\nabla\varphi = (\bm{\beta}\cdot\nabla\psi) \chi\ge
2b_0\chi_*$.  For the terms in $a_h^{rc}(\cdot,\cdot)$ on the interior
facets, using (\ref{eq:weighted-jump}) and (\ref{eq:weighted-jump-t}),
we have
$$
\begin{aligned}
%\langle \mathcal{S}_{\bm{\beta}}\bm{v}_h,\bm{v_h}\varphi\rangle_{\mathcal{F}_h^0}:=
 &\langle \llbracket \bm{\beta}\times \bm{v}_h\rrbracket_{\rm
  t},\vavg{\bm{v}_h\varphi}_{1-\alpha_d}\rangle_{\mathcal{F}_h^0} -
  \langle\left[\bm{v}_h\right]_{\rm{n}},
  \{\bm{\beta}\cdot\bm{v}_h\varphi\}_{1-\alpha}\rangle_{\mathcal{F}_h^0}
  \notag \\
  &~\quad +\langle\llbracket \bm{v}_h\rrbracket_{\rm
  t},\vavg{\bm{\beta}\times
  \bm{v}_h\varphi}_\alpha-\vavg{\bm{\beta}\times
  \bm{v}_h\varphi}_{\alpha_d}\rangle_{\mathcal{F}_h^0}\\
  =&~ \langle \llbracket \bm{\beta}\times \bm{v}_h\rrbracket_{\rm
  t},\vavg{\bm{v}_h\varphi}\rangle_{\mathcal{F}_h^0}-\langle
  \llbracket \bm{\beta}\times \bm{v}_h\rrbracket_{\rm
  t},\llbracket{\bm{v}_h\rrbracket_{\rm
  t}}\times\frac{\llbracket\alpha_d\rrbracket}{2}\varphi\rangle_{\mathcal{F}_h^0}\\
  &~~-\langle\left[\bm{v}_h\right]_{\rm{n}},
  \{\bm{\beta}\cdot\bm{v}_h\varphi\}\rangle_{\mathcal{F}_h^0}
  +\langle\left[\bm{v}_h\right]_{\rm{n}},
  \llbracket\bm{\beta}\cdot\bm{v}_h\rrbracket \cdot
  \frac{\llbracket\alpha\rrbracket}{2}\varphi\rangle_{\mathcal{F}_h^0}
  \\
  &~ ~+\langle\llbracket \bm{v}_h\rrbracket_{\rm
  t},\llbracket{\bm{\beta}\times \bm{v}_h}\rrbracket_{\rm
  t}\times\frac{\llbracket\alpha\rrbracket}{2}\varphi\rangle_{\mathcal{F}_h^0}-\langle\llbracket
  \bm{v}_h\rrbracket_{\rm t},\llbracket{\bm{\beta}\times
  \bm{v}_h}\rrbracket_{\rm
  t}\times\frac{\llbracket\alpha_d\rrbracket}{2}\varphi\rangle_{\mathcal{F}_h^0}
  := \sum_{i=1}^{6}I_i.
\end{aligned}
$$
The estimates of $I_i~(i=1,\ldots,6)$ are given as follows. From \eqref{eq:beta-id},
\begin{equation}\label{eq:infsupbdy}
  I_1+I_3 =-\langle (\bm{\beta}\cdot \bm{n}^+)
  (\bm{v}_h^+-\bm{v}_h^-),\vavg{\bm{v}_h}\varphi\rangle_{\mathcal{F}_h^0}
  =-\frac{1}{2}\langle \bm{\beta}\varphi,
  \llbracket\bm{v}_h\cdot\bm{v}_h\rrbracket\rangle_{\mathcal{F}_h^0}.
\end{equation}
It is obvious that $I_2+I_6=0$. Further, \eqref{eq:alpha-id1} and the upwind choice of
$\alpha$ yield
\begin{equation}\label{eq:infsupstability}
  I_4+I_5 = \frac{1}{2}\langle \bm{\beta}\cdot
  \llbracket\alpha\rrbracket\varphi, |\llbracket
  \bm{v}_h\rrbracket_{F}|^2\rangle_{\mathcal{F}_h^0} 
  \geq \frac{C_{up}}{2} \chi_* \langle |\bm{\beta} \cdot \bm{n}|,
  |\llbracket \bm{v}_h\rrbracket_{F}|^2\rangle_{\mathcal{F}_h^0}.
\end{equation}
The stabilization term gives $\langle \tau_F|\llbracket \alpha -
\alpha_d \rrbracket| [\bm{v}_h]_{\rm n}, [\bm{v}_h \varphi]_{\rm n}
\rangle \geq \chi_* \|\bm{v}_h\|_{\rm n}^2$. Then, combining
\eqref{eq:infsupLbeta}--\eqref{eq:infsupstability}, we obtain
 \begin{equation}\label{eq:infsuprc}
 \begin{aligned}
 & \quad a_h^{rc}(\bm{v}_h,\bm{v}_h\varphi) \\
 & \ge\chi_*\left(
   \|(\bar{\rho}+b_0)^{1/2}\bm{v}_h\|_{0,\Omega}^2 + \|\bm{v}_h\|_{\rm
   n}^2 
 +\frac{\min\{C_{up},1\}}{2}\langle |\bm{\beta} \cdot \bm{n}|,
   |\llbracket
   \bm{v}_h\rrbracket_{F}|^2\rangle_{\mathcal{F}_h}\right).
   \end{aligned}
 \end{equation}
Therefore, \eqref{eq:weightinfsup1} is obtained by combining
(\ref{eq:infsupdiffusion}) and (\ref{eq:infsuprc}). The boundedness
\eqref{eq:weightinfsup2} can be easily obtained by using the
definition of energy norm \eqref{eq:energynorm}, the bound of
$\varphi$ in \eqref{eq:chiinequality} and the condition
\eqref{eq:chicondition}.
\end{proof}

\subsection{Projection} \label{subsec:projection}
In order to obtain the discrete inf-sup condition, we introduce the following projection from
$\bm{L}^2(\mathcal{T}_h)$ onto $\bm{V}_h^k$ satisfying
\begin{subequations}\label{eq:projection}
	\begin{align}
		(\bm{\Pi}_h\bm{u},\bm{v}_h)_T&=(\bm{u},\bm{v}_h)_T \quad & \forall \bm{v}_h\in {\bm{\mathcal{P}}}_{k-1}(T),\label{eq:projectionT}\\
		\langle \bm{\Pi}_h \bm{u}\cdot \bm{n},v_h\rangle_F&=\langle \bm{u}\cdot\bm{n},v_h\rangle_F \quad
		&\forall v_h \in \mathcal{P}_k(F),\forall F\in \partial T\backslash F^\star,\label{eq:projectionF}
	\end{align}
\end{subequations}
where $F^\star$ is an arbitrary facet of $T$ that satisfies
Assumption \ref{as:normal-beta}. We have the following approximation
property for $\bm{\Pi}_h$, see \cite[Prop. 2.1]{cockburn2008optimal} for a proof. 
\begin{lemma}[approximation property] \label{lm:projectionapprox}
	Assume that $\bm{u}\in \bm{H}^{s+1}(T)$ for $s\in [0,k]$ on an element $T\in\mathcal{T}_h$. Then,
	\begin{equation}
		\|\bm{\Pi}_h\bm{u}-\bm{u}\|_{0,T}\le Ch^{s+1}|\bm{u}|_{s+1,T}.
	\end{equation}
\end{lemma}

With this projection, we also need to estimate the difference between
$\bm{v}_h\varphi$ and the corresponding projection
$\bm{\Pi}_h(\bm{v}_h\varphi)\in \bm{V}_h^k$ which is established in
the following lemma.
\begin{lemma}[superconvergence] \label{lm:diffPivphi}
  Let $\varphi\in W^{k+1,\infty}(\Omega)$ be the function defined in
  (\ref{eq:weightfunction}). Then, for any $\bm{v}_h\in \bm{V}_h^k$ and
  $\kappa\in \mathbb{R}^+$,
\begin{subequations}
\begin{align} 
    \|\bm{v}_h\varphi-\bm{\Pi}_h(\bm{v}_h\varphi)\|_{0,\Omega} & \le
    Ch\|\chi\|_{k+1,\infty,\Omega}\|\bm{v}_h\|_{0,\Omega},    \label{eq:phivapprox1} \\
    |\bm{v}_h\varphi-\bm{\Pi}_h(\bm{v}_h\varphi)|_{1,\Omega} & \le
    C\|\chi\|_{k+1,\infty,\Omega}\|\bm{v}_h\|_{0,\Omega},  \label{eq:phivapprox2} \\
  \bigg(	\sum_{F\in
    \mathcal{F}_h}\|\bm{v_h}\varphi -
    \bm{\Pi}_h(\bm{v}_h\varphi)\|_{0,F}^2\bigg)^{1/2} & \le
    Ch^{1/2}\|\chi\|_{k+1,\infty,\Omega}\|\bm{v}_h\|_{0,\Omega},  \label{eq:phivapprox3}
\end{align}
\end{subequations}
where the constant $C$ depends only on $k$, shape-regularity constant, and $\Omega$.
\end{lemma}
\begin{proof}
  The proof of this lemma is straightforward through the results in
  Lemma \ref{lm:projectionapprox}, trace inequality and inverse
  inequality. We refer the readers to 
  \cite[Lemma 4.2]{ayuso2009discontinuous} for a detailed proof.
\end{proof}

We are now in the position to give the following results.
\begin{lemma}[projection difference] \label{lm:diffestimate}
	With the hypothesis of Lemma \ref{lm:weightinfsup}, there exist two positive constants $\chi_1,\chi_2$ depending on $\chi,b_0,\theta,\alpha,\alpha_d$ and $\|\bm{\beta}\|_{1,\infty,\Omega}$, such that for any value of $\kappa$, the corresponding $\varphi$ verifies
\begin{subequations} \label{eq:projection-difference}
\begin{align}
	a_h^d(\bm{v}_h,\bm{v}_h\varphi-\bm{\Pi}_h(\bm{v}_h\varphi)) &\le \chi_1\triplenorm{\bm{v}_h}_{d}\triplenorm{\bm{v}_h}_{rc}, \quad \forall\bm{v}_h\in \bm{V}_h^k, \label{eq:diffestimatediff}\\
	a_h^{rc}(\bm{v}_h,\bm{v}_h\varphi-\bm{\Pi}_h(\bm{v}_h\varphi)) &\le \chi_2h^{1/2}\triplenorm{\bm{v}_h}_{rc}^2, \quad \forall\bm{v}_h\in \bm{V}_h^k.\label{eq:diffestimaterc}
\end{align}
\end{subequations}
\end{lemma}
	
\begin{proof}
To simplify the notation, we write $\bm{\xi}:= \bm{v}_h\varphi-\bm{\Pi}_h(\bm{v}_h\varphi)$.
\paragraph{Step 1: diffusive part} Using inverse inequality, trace inequality and (\ref{eq:phivapprox2})-(\ref{eq:phivapprox3}), we have
\begin{equation}\label{eq:diffencediffusion}
	\begin{aligned}
		\quad &~ a^d_h(\bm{v}_h,\bm{\xi})\\
		\le&~ \varepsilon\|\nabla_h\times \bm{v}_h\|_{0,\Omega}\| \nabla_h\times\bm{\xi}\|_{0,\Omega}+\varepsilon|\theta|\alpha_d^{\max}\big(\sum_{F\notin \Gamma_N} h_F\|\nabla_h\times \bm{\xi}\|_{0,F}^2\big)^{1/2}\|\bm{v}_h\|_{\rm t}\\
		&\quad+\varepsilon\alpha_d^{\max}\big(\sum_{F\notin \Gamma_N} h_F\|\nabla_h\times \bm{v}_h\|_{0,F}^2\big)^{1/2}\|\xi\|_{\rm t}+\varepsilon\eta^{\max} \|\bm{v}_h\|_{\rm t}\|\xi\|_{\rm t}\\
%		\le &~C\varepsilon\|\chi\|_{k+1,\infty,\Omega}\big[(1+\alpha_d^{\max})\|\nabla_h\times \bm{v}_h\|_{0,\Omega}+(|\theta|\alpha_d^{\max}+\eta^{\max})\|\bm{v}_h\|_{\rm t}\big]\|\bm{v}_h\|_{0,\Omega}\\
%		&+C\eta^{max}\varepsilon\|\chi\|_{k+1,\infty,\Omega}\|\bm{v}_h\|_{\rm t}\|\bm{v}_h\|_{0,\Omega}
		\le&~C\varepsilon^{1/2}\|\chi\|_{k+1,\infty,\Omega}\triplenorm{\bm{v}_h}_{d}\triplenorm{\bm{v}_h}_{rc},
	\end{aligned}
\end{equation}
where $\eta^{\max}=\sup_{F\notin \Gamma_N}\eta_F$ and (\ref{eq:diffencediffusion}) follows with $\chi_1=C\|\chi\|_{k+1,\infty,\Omega}$ since $\varepsilon$ is a bounded constant.
\paragraph{Step 2: reactive-convective part}
Let $\bm{P}_h^0\bm{\beta}$ be the piecewise $L^2$-projection of $\bm{\beta}$ onto constants. By definition of $\bm{\Pi}_h$ (\ref{eq:projectionT}), it holds that
$$
( -\bm{P}_h^0\bm{\beta}\times (\nabla_h\times\bm{v}_h)+(\nabla_h \bm{v}_h)^T\bm{P}_h^0\bm{\beta},\bm{\xi})_T =0 \quad \forall T \in \mathcal{T}_h,
$$
since the leftmost component belongs to $\bm{\mathcal{P}}_{k-1}(T)$. Then, we have
\begin{equation}\label{eq:rcdiff}
\begin{aligned}
	&~~ a^{rc}_h (\bm{v}_h,\bm{\xi})\\
	=&~([\gamma+(\nabla\bm{\beta})^T]\bm{v}_h,\bm{\xi})+( -[\bm{\beta}-\bm{P}_h^0\bm{\beta}]\times (\nabla_h \times\bm{v}_h)+(\nabla_h \bm{v}_h)^T[\bm{\beta}-\bm{P}_h^0\bm{\beta}],\bm{\xi})\\
	&~ +\langle \llbracket \bm{\beta}\times \bm{v}_h\rrbracket_{\rm t},\vavg{\bm{\xi}}\rangle_{\mathcal{F}_h^0}-\langle \llbracket \bm{\beta}\times \bm{v}_h\rrbracket_{\rm t},\llbracket{\bm{\xi}\rrbracket_{\rm t}}\times\frac{\llbracket\alpha_d\rrbracket}{2}\rangle_{\mathcal{F}_h^0}\\
	&~ -\langle\left[\bm{v}_h\right]_{\rm{n}},
	\{\bm{\beta}\cdot\bm{\xi}\}\rangle_{\mathcal{F}_h^0} +\langle\left[\bm{v}_h\right]_{\rm{n}},
	\llbracket\bm{\beta}\cdot\bm{\xi}\rrbracket\cdot\frac{\llbracket\alpha\rrbracket}{2}\rangle_{\mathcal{F}_h^0} \\
	&~ +\langle\llbracket \bm{v}_h\rrbracket_{\rm t},\llbracket{\bm{\beta}\times \bm{\xi}}\rrbracket_{\rm t}\times\frac{\llbracket\alpha\rrbracket}{2}\rangle_{\mathcal{F}_h^0}-\langle\llbracket \bm{v}_h\rrbracket_{\rm t},\llbracket{\bm{\beta}\times \bm{\xi}}\rrbracket_{\rm t}\times\frac{\llbracket\alpha_d\rrbracket}{2}\rangle_{\mathcal{F}_h^0}\\
	&~ +\langle \tau_F|\llbracket\alpha-\alpha_d\rrbracket|[\bm{v}_h]_{\rm{n}},[\bm{\xi}]_{\rm{n}}\rangle_{\mathcal{F}_h^0}  -\langle \bm{\beta}\cdot\bm{n},\bm{v}_h\cdot\bm{\xi}\rangle_{\mathcal{F}_{h,D}^{\partial,-} \cup \mathcal{F}_{h,N}^{\partial,-} }\notag 
	=: \sum_{i=1}^{10} J_i.
\end{aligned}
\end{equation}
The estimates of $J_i~(i=1,\cdots,10)$ are given as follows. 

\noindent{\it Step 2-1: terms without weight}. Firstly, (\ref{eq:phivapprox1}) implies
\begin{equation}\label{eq:approxJ1}
J_1\le Ch\|\chi\|_{k+1,\infty,\Omega}\|\bm{v}_h\|_{0,\Omega}^2.
\end{equation}
The approximation property of $\bm{P}_h^0$, (\ref{eq:phivapprox2}) and the inverse inequality give
\begin{equation}\label{eq:approxJ2}
	J_2\le Ch^2\|\chi\|_{k+1,\infty,\Omega}|\bm{\beta}|_{1,\infty,\Omega}|\bm{v}_h|_{1,\Omega}\|\bm{v}_h\|_{0,\Omega}\le Ch\|\chi\|_{k+1,\infty,\Omega}|\bm{v}_h\|_{0,\Omega}^2.
\end{equation}
Using \eqref{eq:beta-id} and \eqref{eq:phivapprox3}, we deduce that 
\begin{equation}\label{eq:approxJ3510}
	\begin{aligned}
	J_3+J_5+J_{10}=&-\langle (\bm{\beta}\cdot \bm{n}^+)\llbracket\bm{v}_h\rrbracket_{F},\vavg{\bm{\xi}}\rangle_{\mathcal{F}_h^0}-\langle \bm{\beta}\cdot\bm{n},\bm{v}_h\cdot\bm{\xi}\rangle_{\mathcal{F}_{h,D}^{\partial,-} \cup \mathcal{F}_{h,N}^{\partial,-} }\\
	\le  &~Ch^{1/2}\|\chi\|_{k+1,\infty,\Omega}\big(\sum_{F\in \mathcal{F}_h}\||\bm{\beta}\cdot\bm{n}|^{1/2}\llbracket\bm{v}_h\rrbracket_{F}\|_{0,F}^2\big)^{1/2}\|\bm{v}_h\|_{0,\Omega}.
	\end{aligned}
\end{equation}

\noindent{\it Step 2-2: weighed terms}. Further, \eqref{eq:alpha-id1} and \eqref{eq:alpha-id2} leads to  
\begin{equation}\label{eq:approxJrest}
\begin{aligned}
&~ J_4+J_6+J_7+J_8\\
=&~\frac{1}{2}\langle(\bm{\beta}\cdot\llbracket\alpha\rrbracket)\llbracket\bm{v}_h\rrbracket_{F},\llbracket\bm{\xi}\rrbracket_{F}\rangle_{\mathcal{F}_h^0}+\langle\frac{[\alpha-\alpha_d]_F}{2}\bm{\beta}\cdot \llbracket\bm{\xi}\rrbracket_{F} ,[\bm{v}_h]_{\rm n}\rangle_{\mathcal{F}_h^0}\\
&-\langle\frac{[\alpha-\alpha_d]_F}{2}\bm{\beta}\cdot \llbracket\bm{v}_h\rrbracket_{F}, [\bm{\xi}]_{\rm n}\rangle_{\mathcal{F}_h^0} =:K_1+K_2+K_3.
\end{aligned}
\end{equation}
In light of definition of $\triplenorm{\cdot}_{rc}$ and \eqref{eq:phivapprox3}, we get
\begin{align}
	K_1 & \le Ch^{1/2}\|\chi\|_{k+1,\infty,\Omega}\big(\sum_{F\in \mathcal{F}_h}\||\bm{\beta}\cdot\bm{n}|^{1/2}\llbracket\bm{v}_h\rrbracket_{F}\|_{0,F}
	^2\big)^{1/2}\|\bm{v}_h\|_{0,\Omega}, \label{eq:approxK1} \\
	K_2+J_9 & \le Ch^{1/2}\|\chi\|_{k+1,\infty,\Omega}\|\bm{v}_h\|_{\rm n}\|\bm{v}_h\|_{0,\Omega}. \label{eq:approxK2J9}
\end{align}

We now consider the estimate of the most difficult term $K_3$, where $\bm{\beta} \cdot \llbracket \bm{v}_h \rrbracket_F$ does not appear in $\triplenorm{\cdot}_{rc}$ so that the simple argument by Cauchy-Schwarz inequality does not work.  Note that $\mathcal{F}_h^0$ can be divided into $\mathcal{F}_h^\star$ and $\mathcal{F}_h^0 \setminus \mathcal{F}_h^\star$, where the projection $\bm{\Pi}_h$ on the latter facet preserves the normal moments up to polynomial of order $k$, due to \eqref{eq:projectionF}. Therefore, we have
$$
-\langle (\bm{P}_{\mathcal{F},h}^{0}\bm{\beta})\cdot \llbracket\bm{v}_h\rrbracket_{F}, [\bm{\xi}]_{\rm n}\rangle_{\mathcal{F}_h^0/\mathcal{F}_h^\star}=0,
$$
where $\bm{P}_{\mathcal{F},h}^0$ is the $L^2$-projection in the piecewise constant space of $\mathcal{F}_h$. Then, we have
\begin{equation}\label{K3_3}
\begin{aligned}
& -\langle\frac{[\alpha-\alpha_d]_F}{2}\bm{\beta}\cdot \llbracket\bm{v}_h\rrbracket_{F}, [\bm{\xi}]_{\rm n}\rangle_{\mathcal{F}_h^0/\mathcal{F}_h^\star}\\
		=&-\langle\frac{[\alpha-\alpha_d]_F}{2}(\bm{\beta}-\bm{P}_{\mathcal{F},h}^0\bm{\beta})\cdot \llbracket\bm{v}_h\rrbracket_{F}, [\bm{\xi}]_{\rm n}\rangle_{\mathcal{F}_h^0/\mathcal{F}_h^\star}\\
		\le&~ Ch|\bm{\beta}|_{1,\infty,\Omega}\|\chi\|_{k+1,\infty,\Omega}\|\bm{v}_h\|_{0,\Omega}^2,
	\end{aligned}
\end{equation}
where the approximation property of $\bm{P}_{\mathcal{F},h}^0$, trace inequality, inverse inequality and \eqref{eq:phivapprox3} are used. Further, using  Assumption \ref{as:normal-beta} (normal domination), we obtain
\begin{equation}\label{eq:approxK3_1}
	\begin{aligned}
	&~ -\langle\frac{[\alpha-\alpha_d]_F}{2}\bm{\beta}\cdot \llbracket\bm{v}_h\rrbracket_{F}, [\bm{\xi}]_{\rm n}\rangle_{\mathcal{F}_h^\star}\\
	\le&~ Ch^{1/2}\|\chi\|_{k+1,\infty,\Omega}\big(\sum_{F\in \mathcal{F}_h^\star}\|\bm{\beta}\cdot\llbracket \bm{v}_h\rrbracket_F\|_{0,F}^2\big)^{1/2}\|\bm{v}_h\|_{0,\Omega}\\
	\le&~ Ch^{1/2}\|\chi\|_{k+1,\infty,\Omega}\big(\sum_{F\in \mathcal{F}_h^\star}\||\bm{\beta}\cdot\bm{n}|^{1/2} \llbracket\bm{v}_h\rrbracket_F\|_{0,F}^2\big)^{1/2}\|\bm{v}_h\|_{0,\Omega}.
	\end{aligned}
\end{equation}
Finally, collecting \eqref{eq:approxJ1}-\eqref{eq:approxK3_1}, we then arrive at
$$
a_h^{rc}(\bm{v}_h,\bm{\xi})\le Ch^{1/2}\|\chi\|_{k+1,\infty,\Omega}\triplenorm{\bm{v}_h}_{rc}\|\bm{v}_h\|_{0,\Omega}\le C h^{1/2}\|\chi\|_{k+1,\infty,\Omega}\triplenorm{\bm{v}_h}_{rc}^2,
$$
namely, (\ref{eq:diffestimaterc}) with $\chi_2=C \|\chi\|_{k+1,\infty,\Omega}$.
\end{proof}

\begin{remark}[compared to the scalar case] We note that the standard $L^2$ projection is enough for the scalar case \cite{ayuso2009discontinuous}. For the vector case, however, we need the tailored projection $\bm{\Pi}_h$ \eqref{eq:projection} to deal with the facets on which the normal component of $\bm{\beta}$ is not dominated. 
\end{remark}

The next theorem provides the stability result for the variational form \eqref{eq:variational-form}.
\begin{theorem}[inf-sup condition] \label{thm:infsup}
	With the hypothesis of Lemma \ref{lm:weightinfsup}, there exist positive constants $\zeta$ and $h_0$ independent of $h$, such that, for $h<h_0$,
	\begin{equation}\label{eq:infsup}
		\sup_{\bm{v}_h\in \bm{V}_h^k}\frac{a_h(\bm{u}_h,\bm{v}_h)}{\triplenorm{\bm{v}_h}}\ge \zeta \triplenorm{\bm{u}_h}, \quad \forall \bm{u}_h\in \bm{V}_h^k.
	\end{equation}

\end{theorem}
\begin{proof}	
	For $\bm{u}_h\in \bm{V}_h^k$, let $\bm{v}_h := \bm{\Pi}_h(\bm{u}_h\varphi)\in\bm{V}_h^k$ with the tailored projection $\bm{\Pi}_h$ defined in \eqref{eq:projection}. We only need to prove that
	\begin{equation}
	a_h(\bm{u}_h,\bm{v}_h)\ge c_2\triplenorm{\bm{u}_h}^2, \quad 
		\triplenorm{\bm{v}_h}\le c_1\triplenorm{\bm{u}_h},
\end{equation}
for some positive constants $c_1$ and $c_2$. From \eqref{eq:weightinfsup1} and \eqref{eq:projection-difference}, we have
$$
\begin{aligned}
	a_h(\bm{u}_h,\bm{\Pi}_h(\bm{u}_h\varphi))
	=&~a_h(\bm{u}_h,\bm{\Pi}_h(\bm{u}_h\varphi)-\bm{u}_h\varphi)+a_h(\bm{u}_h,\bm{u}_h\varphi)\\
	\ge&~ -\chi_1\triplenorm{\bm{u}_h}_{d}\triplenorm{\bm{u}_h}_{rc}-\chi_2h^{1/2}\triplenorm{\bm{u}_h}^2_{rc}+\frac{\chi_*+\kappa}{2} \triplenorm{\bm{u}_h}_{d}^2\\
	&~\quad+\frac{\chi_*\min\{C_{up},1\}}{2}\triplenorm{\bm{u}_h}^2_{rc}.
\end{aligned}
$$
If we choose $\kappa$  and $h_0$ that satisfy $\chi_*+\kappa\ge \frac{8\chi_1^2}{\chi_*\min\{C_{up},1\}}$ and $h<h_0 :=(\frac{\chi_*\min\{C_{up},1\}}{8\chi_2})^2$ respectively, we obtain
$$
a_h(\bm{u}_h,\bm{\Pi}_h(\bm{u}_h\varphi))\ge \frac{\chi_*\min\{C_{up},1\}}{4}\triplenorm{\bm{u}_h}^2.
$$
From \eqref{eq:energynorm}, Lemma \ref{lm:diffPivphi} and \eqref{eq:weightinfsup2},  we easily deduce
$$
%\begin{aligned}
\triplenorm{\bm{\Pi}_h(\bm{u}_h\varphi)}\le C\triplenorm{\bm{u}_h\varphi} \leq c_1 \triplenorm{\bm{u}_h}.
%& ~C\max\{\frac{\|\chi\|_{k+1,\infty,\Omega}}{\chi_*+\kappa},1\}\triplenorm{\bm{u}_h\varphi}\\
%\le& ~C \max\{\frac{\|\chi\|_{k+1,\infty,\Omega}(\chi^*+\kappa)}{\chi_*+\kappa},\chi^*+\kappa\}\triplenorm{\bm{u}_h}\\
%\le & ~C\max\{\|\chi\|_{k+1,\infty,\Omega},\chi^*+\kappa\}\triplenorm{\bm{u}_h}.
%\end{aligned}
$$
Thus \eqref{eq:infsup} follows.
\end{proof}

\section{A priori error estimate} \label{sc:error}
In this section, we show a priori error estimate in the norm (\ref{eq:energynorm}). Invoking the projection $\bm{\Pi}_h$ defined in (\ref{eq:projection}) and the approximation property in Lemma \ref{lm:projectionapprox}, we have the following property:
\begin{equation}\label{eq:projectinapproxT}
	\|\bm{u}-\bm{\Pi}_h\bm{u}_h\|_{r,T}\le Ch^{k+1-r}|\bm{u}|_{k+1,T},\quad r=0,1,2,\quad T\in\mathcal{T}_h.
\end{equation}
Further, using trace inequality, we deduce that
\begin{equation}\label{eq:projectionapproxF}
		\|\bm{u}-\bm{\Pi}_h\bm{u}_h\|_{0,F}\le Ch_T^{k+1/2}|\bm{u}|_{k+1,T},\quad \forall F\in\mathcal{F}_h \text{ and }F\subset \partial T.
\end{equation}
Then we are able to give the following error estimate.
\begin{theorem}[error estimate] \label{tm:estimate}
	Let $\bm{u}\in\bm{H}^{k+1}(\Omega)$ be the solution of (\ref{eq:Hcurl-cd}), and let $\bm{u}_h$ be the solution of the discrete problem (\ref{eq:variational-form}). There exists a constant $C>0$ depending on the domain $\Omega$, the shape-regularity constant, and $\bm{\beta},\gamma,\theta,\alpha,\alpha_d$ such that
	\begin{equation}\label{eq:err}
		\triplenorm{\bm{u}-\bm{u}_h}\le Ch^k(\varepsilon^{1/2}+h^{1/2})|\bm{u}|_{k+1,\Omega}.
	\end{equation}
\end{theorem}
\begin{proof}
We define $\bm{e}_h:=\bm{u}-\bm{\Pi}_h\bm{u},\quad \bm{\varepsilon}_h:=\bm{u}_h-\bm{\Pi}_h\bm{u}$. Thanks to Theorem \ref{thm:infsup} (inf-sup condition) and consistency (\ref{eq:consistency}), we have
\begin{equation}\label{eq:deltaerr}
  \zeta\triplenorm{\bm{\varepsilon}_h} \le \sup_{\bm{v}_h \in \bm{V}_h^k}  \frac{a_h(\bm{\varepsilon}_h,\bm{v}_h)}{\triplenorm{\bm{v}_h}}
   = \sup_{\bm{v}_h \in \bm{V}_h^k} \frac{a_h(\bm{e}_h,\bm{v}_h)}{\triplenorm{\bm{v}_h}}.
\end{equation}
The estimate of $a_h(\bm{e}_h, \bm{v}_h)$ is divided into the following steps. 

\noindent{\it Step 1: diffusive part}. The estimate of the diffusive part is standard and similar to \eqref{eq:diffencediffusion} by using local trace inequality, 
inverse inequality and \eqref{eq:projectinapproxT}--\eqref{eq:projectionapproxF}:
\begin{equation}\label{eq:differr}
a^{d}_h(\bm{e}_h,\bm{v}_h) \le C\varepsilon^{1/2}h^k\triplenorm{\bm{v}_h}_{d}|\bm{u}|_{k+1,\Omega}.
%\begin{aligned}
%& ~~ a^{d}_h(\bm{e}_h,\bm{v}_h)\\
% \le&~ \varepsilon \|\nabla_h\times \bm{e}_h\|_{0,\Omega}\| \nabla_h\times\bm{v}_h\|_{0,\Omega} 
% + \varepsilon|\theta|\alpha_d^{\max}\big(\sum_{F\notin \Gamma_N} h_F\|\nabla_h\times \bm{v}_h\|_{0,F}^2\big)^{1/2}\|\bm{e}_h\|_{\rm t}\\
%& +\varepsilon\alpha_d^{\max}\big(\sum_{F\notin \Gamma_N} h_F\|\nabla_h\times \bm{e}_h\|_{0,F}^2\big)^{1/2}\|\bm{v}_h\|_{\rm t}
%   +\varepsilon\eta^{\max} \|\bm{e}_h\|_{\rm t}\|\bm{v}_h\|_{\rm t}\\
%\le &~C\varepsilon h^k\big((1+|\theta|\alpha_d^{\max})\|\nabla_h\times \bm{v}_h\|_{0,\Omega}+(\alpha_d^{max}+\eta^{max})\|\bm{v}_h\|_{\rm t}\big)|\bm{u}|_{k+1,\Omega}\\
%&+C\eta^{max}\varepsilon\|\chi\|_{k+1,\infty,\Omega}\|\bm{v}_h\|_{\rm t}\|\bm{v}_h\|_{0,\Omega}
%\le&~{C}\varepsilon^{1/2}h^k\triplenorm{\bm{v}_h}_{d}|\bm{u}|_{k+1,\Omega}.
%\end{aligned}
\end{equation}

\noindent{\it Step 2: reactive-convective part}. Invoking \eqref{eq:Lie-2}, we obtain
$$
\begin{aligned}
(L_{\bm{\beta},h}\bm{e}_h,\bm{v}_h) &= (\bm{e}_h,\mathcal{L}_{\bm{\beta},h} \bm{v}_h)+\langle\bm{\beta},\llbracket\bm{e}_h\cdot\bm{v}_h\rrbracket\rangle_{\mathcal{F}_h}\\
%	&=(\nabla \times ({\bm{\beta}} \times \bm{v}_h) - {\bm{\beta}} \nabla \cdot \bm{v}_h,\bm{e}_h)+\langle\bm{\beta},\llbracket\bm{e}_h\cdot\bm{v}_h\rrbracket\rangle_{\mathcal{F}_h}\\
	&=((\nabla\bm{\beta})^T\bm{v}_h-(\nabla\cdot\bm{\beta})\bm{v_h},\bm{e}_h)-((\nabla_h \bm{v}_h)^T\bm{\beta},\bm{e}_h)+\langle\bm{\beta},\llbracket\bm{e}_h\cdot\bm{v}_h\rrbracket\rangle_{\mathcal{F}_h}.
\end{aligned}
$$
Since $(\nabla_h \bm{v}_h)^T\bm{P}_h^0\bm{\beta}\in \bm{V}_h^{k-1}$, the definition of projection $\bm{\Pi}_h$ \eqref{eq:projectionT} leads to 
$$
((\nabla_h\bm{v}_h)^T\bm{P}_h^0\bm{\beta},\bm{e}_h)_T =0 \quad \forall T \in \mathcal{T}_h.
$$
Hence, we can rewrite the convective-reactive part as
\begin{equation*}\label{eq:rcerr}
	\begin{aligned}
		& \quad a^{rc}_h(\bm{e}_h,\bm{v}_h)\\
		&= ((\nabla\bm{\beta})^T\bm{v}_h-(\nabla\cdot\bm{\beta})\bm{v_h}+\gamma\bm{v}_h,\bm{e}_h)-((\nabla_h\bm{v}_h)^T(\bm{\beta}-\bm{P}_h^0\bm{\beta}),\bm{e}_h)\\
		&\quad+\langle \llbracket \bm{\beta}\times \bm{e}_h\rrbracket_{\rm t},\vavg{\bm{v}_h}\rangle_{\mathcal{F}_h^0}-\langle \llbracket \bm{\beta}\times \bm{e}_h\rrbracket_{\rm t},\llbracket{\bm{v}_h\rrbracket_{\rm t}}\times\frac{\llbracket\alpha_d\rrbracket}{2}\rangle_{\mathcal{F}_h^0}\\
		&\quad -\langle\left[\bm{e}_h\right]_{\rm{n}},
		\{\bm{\beta}\cdot\bm{v}_h\}\rangle_{\mathcal{F}_h^0} +\langle\left[\bm{e}_h\right]_{\rm{n}},
		\llbracket\bm{\beta}\cdot\bm{v}_h\rrbracket\cdot\frac{\llbracket\alpha\rrbracket}{2}\rangle_{\mathcal{F}_h^0} \\
		&\quad +\langle\llbracket \bm{e}_h\rrbracket_{\rm t},\llbracket{\bm{\beta}\times \bm{v}_h}\rrbracket_{\rm t}\times\frac{\llbracket\alpha\rrbracket}{2}\rangle_{\mathcal{F}_h^0}-\langle\llbracket \bm{e}_h\rrbracket_{\rm t},\llbracket{\bm{\beta}\times \bm{v}_h}\rrbracket_{\rm t}\times\frac{\llbracket\alpha_d\rrbracket}{2}\rangle_{\mathcal{F}_h^0}\\
		&\quad+\langle \tau_F|\llbracket\alpha-\alpha_d\rrbracket|[\bm{e}_h]_{\rm{n}},[\bm{v}_h]_{\rm{n}}\rangle_{\mathcal{F}_h^0} \\
		&\quad-\langle \bm{\beta}\cdot\bm{n},\bm{e}_h\cdot\bm{v}_h\rangle_{\mathcal{F}_{h,D}^{\partial,-} \cup \mathcal{F}_{h,N}^{\partial,-} } 
		+\langle\bm{\beta},\llbracket\bm{e}_h\cdot\bm{v}_h\rrbracket\rangle_{\mathcal{F}_h} :=  \sum_{i=1}^{11} \tilde{J}_i.
	\end{aligned}
\end{equation*}
The estimates of $\tilde{J}_i~(i=1,\cdots,11)$ are given as follows. 

\noindent{\it Step 2-1: terms without weight}. Using the approximation properties of $\bm{\Pi}_h$ \eqref{eq:projectinapproxT} and \eqref{eq:projectionapproxF}, the inverse inequality, we obtain
\begin{align}
	\tilde{J}_1 &\le Ch^{k+1}\|\bm{v}_h\|_{0,\Omega}|\bm{u}|_{k+1,\Omega}, \label{eq:errJ1} \\
	\tilde{J}_2 &\le Ch^{k+2}|\bm{\beta}|_{1,\infty}|\bm{v}_h|_{1,\Omega}|\bm{u}|_{k+1,\Omega}\le Ch^{k+1}\|\bm{v}_h\|_{0,\Omega}|\bm{u}|_{k+1,\Omega}. \label{eq:errJ2}
\end{align}
Using \eqref{eq:beta-id} and \eqref{eq:projectionapproxF}, we deduce
\begin{equation}\label{eq:errJ351011}
\begin{aligned}
	& \quad \tilde{J}_3+\tilde{J}_5+\tilde{J}_{10}+\tilde{J}_{11}\\
	&=-\langle \bm{\beta}\cdot \bm{n}^+\llbracket\bm{e}_h\rrbracket_{F},\vavg{\bm{v}_h}\rangle_{\mathcal{F}_h^0}+\langle\bm{\beta},\llbracket\bm{e}_h\cdot\bm{v}_h\rrbracket\rangle_{\mathcal{F}_h^0}+\langle \bm{\beta}\cdot\bm{n},\bm{v}_h\cdot\bm{e}_h\rangle_{\mathcal{F}_{h,D}^{\partial,+} \cup \mathcal{F}_{h,N}^{\partial,+} }\\
		&= \langle \bm{\beta}\cdot\bm{n}^+\llbracket\bm{v}_h\rrbracket_{F},\vavg{\bm{e}_h}\rangle_{\mathcal{F}_h^0}+\langle \bm{\beta}\cdot\bm{n},\bm{v}_h\cdot\bm{e}_h\rangle_{\mathcal{F}_{h,D}^{\partial,+} \cup \mathcal{F}_{h,N}^{\partial,+} }\\
		&\le Ch^{k+1/2}\big(\sum_{F\in \mathcal{F}_h}\||\bm{\beta}\cdot\bm{n}|^{1/2}\llbracket\bm{v}_h\rrbracket_{F}\|_{0,F}^2\big)^{1/2}|\bm{u}|_{k+1,\Omega}.
	\end{aligned}
\end{equation}

\noindent{\it Step 2-2: weighed terms}. Further, \eqref{eq:alpha-id1} and \eqref{eq:alpha-id2} lead to  
\begin{equation}\label{eq:errJ4678}
\begin{aligned}
&~\tilde{J}_4+\tilde{J}_6+\tilde{J}_7+\tilde{J}_8\\
		=&~\frac{1}{2}\langle\bm{\beta}\cdot\llbracket\alpha\rrbracket\llbracket\bm{e}_h\rrbracket_{F},\llbracket\bm{v}_h\rrbracket_{F}\rangle_{\mathcal{F}_h^0}-\langle\frac{[\alpha-\alpha_d]_F}{2}\bm{\beta}\cdot \llbracket\bm{e}_h\rrbracket_{F} ,[\bm{v}_h]_{\rm n}\rangle_{\mathcal{F}_h^0}\\
		&~+\langle\frac{[\alpha-\alpha_d]_F}{2}\bm{\beta}\cdot \llbracket\bm{v}_h\rrbracket_{F} ,[\bm{e}_h]_{\rm n}\rangle_{\mathcal{F}_h^0} =:~\tilde{K}_1+\tilde{K}_2+\tilde{K}_3.
	\end{aligned}
\end{equation}
Similar to \eqref{eq:approxK1} and \eqref{eq:phivapprox3}, we have 
\begin{align}
	\tilde{K}_1 & \le Ch^{k+1/2}\big(\sum_{F\in \mathcal{F}_h}\||\bm{\beta}\cdot\bm{n}|^{1/2}\llbracket\bm{v}_h\rrbracket_{F}\|_{0,F}^2\big)^{1/2}|\bm{u}|_{k+1,\Omega}, \label{eq:errK1}\\
	\tilde{K}_2+\tilde{J}_9 & \le Ch^{k+1/2}\|\bm{v}_h\|_{\rm n}|\bm{u}|_{k+1,\Omega}. \label{eq:errK2}
\end{align}
The estimate of $\tilde{K}_3$ can also be divided into $\mathcal{F}_h^\star$ and $\mathcal{F}_h^0 \setminus \mathcal{F}_h^\star$, i.e.
\begin{subequations} \label{eq:errK3}
\begin{equation} \label{eq:errK3-1}
\begin{aligned}
& ~  \langle\frac{[\alpha-\alpha_d]_F}{2}\bm{\beta}\cdot \llbracket\bm{v}_h\rrbracket_{F}, [\bm{e}_h]_{\rm n}\rangle_{\mathcal{F}_h^0/\mathcal{F}_h^\star}\\
   =&~\langle\frac{[\alpha-\alpha_d]_F}{2}(\bm{\beta}- \bm{P}_{\mathcal{F},h}^0\bm{\beta})\cdot \llbracket\bm{v}_h\rrbracket_{F}, [\bm{e}_h]_{\rm n}\rangle_{\mathcal{F}_h^0/\mathcal{F}_h^\star}
   \le Ch^{k+1}\|v_h\|_{0,\Omega}|\bm{u}|_{k+1,\Omega},
\end{aligned}
\end{equation}
and 
\begin{equation}\label{eq:errK3-2}
\begin{aligned}
& ~\langle\frac{[\alpha-\alpha_d]_F}{2}\bm{\beta}\cdot \llbracket\bm{v}_h\rrbracket_{F} ,[\bm{e}_h]_{\rm n}\rangle_{\mathcal{F}_h^\star} \\
\le& ~Ch^{k+1/2}\big(\sum_{F\in \mathcal{F}_h^\star}\||\bm{\beta}\cdot\bm{n}|^{1/2} \llbracket\bm{v}_h\rrbracket_F\|_{0,F}^2\big)^{1/2}|\bm{u}|_{k+1,\Omega}.
\end{aligned}
\end{equation}
\end{subequations}
Collecting \eqref{eq:errJ1}-\eqref{eq:errK3-2}, we then arrive at
$$
a_h^{rc}(\bm{e}_h,\bm{v}_h) \le Ch^{k+1/2}\triplenorm{\bm{v}_h}_{rc}|\bm{u}|_{k+1,\Omega}.
$$

In summary, by combing the estimates of diffusive and reactive-convective parts, we obtain
$$
a_h(\bm{e}_h,\bm{v}_h)\le Ch^k(\varepsilon^{1/2}+h^{1/2})\triplenorm{\bm{v}_h}|\bm{u}|_{k+1,\Omega},
$$
which leads to the desired estimate \eqref{eq:err} by using \eqref{eq:deltaerr}, approximation properties of $\bm{\Pi}_h$ \eqref{eq:projectinapproxT}-\eqref{eq:projectionapproxF}, and the triangle inequality. 
\end{proof}
%% end of file

\section{Numerical Experiments} \label{sc:numerical}
In this section, we present several tests for both two-dimensional and three-dimensional problems to verify our theoretical results as well as to display the performance of the proposed DG method in the presence of interior or boundary layers. 
We set the computational domain to be $(0,1)^d$ for $d=2,3$, where the uniform meshes with different mesh sizes are applied. In the case of translational symmetry, known as the transversal electric setting in computational electromagnetics, \eqref{eq:Hcurl-cd} can be reduced to the following two-dimensional (2D) boundary value problem on a domain $\Omega\subset\mathbb{R}^2$ \cite{heumann2013stabilized}
\begin{equation} \label{eq:Hcurl-cd-2d}
	\left\{
	\begin{aligned}
		{\mathbf{R}}\nabla  (\varepsilon \nabla \cdot  ({\rm \mathbf{R}}\bm{u}) ) - {\mathbf{R}}\bm{\beta} 
		\nabla \cdot ({\rm \mathbf{R}} \bm{u}) + \nabla (\bm{\beta} \cdot \bm{u}) +  \gamma
		\bm{u}&= \bm{f} ~~\quad \text{in }\Omega, \\
		({\mathbf{R}}\bm{n} \cdot\bm{u}){\mathbf{R}}\bm{n} + \chi_{\Gamma_D^-} (\bm{u} \cdot \bm{n}) \bm{n} &= \bm{g}_D \quad \text{on } \Gamma_D, \\
		\varepsilon \nabla \cdot({\mathbf{R}} \bm{u}) {\mathbf{R}}\bm{n}+ \chi_{\Gamma_N^-} (\bm{\beta} \cdot \bm{n}) \bm{u} &= \bm{g}_N \quad \text{on } \Gamma_N,
	\end{aligned}
	\right.
\end{equation}
with the $\frac{\pi}{2}$-rotation matrix $\mathbf{R}=\left(\begin{matrix}0 & 1 \\ -1 & 0\end{matrix}\right)$. 

In the 2D case, the flux $\sigma(\bm{u}) := \varepsilon \nabla \cdot (\mathbf{R}\bm{u}) - \bm{\beta} \cdot (\mathbf{R}\bm{u})$ becomes a scalar. Hence, the tangential jump conditions in the strong form \eqref{eq:curl-scheme-B} will be replaced by 
$$
[\bm{u}]_{\rm t} = 0, \quad \llbracket \sigma(\bm{u}) \rrbracket_{\rm t} = \bm{0} \quad \text{on each } F \in \mathcal{F}_h^0, 
$$
where 
$$
[\bm{u}]_{\rm t} := \mathbf{R}\bm{n}^+ \cdot \bm{u}^+ + \mathbf{R}\bm{n}^- \cdot \bm{u}^-, \quad 
\llbracket w \rrbracket_{\rm t} := \mathbf{R}\bm{n}^+w^+ + \mathbf{R}\bm{n}^-w^-.
$$ 
The general choice of $\mathcal{B}$'s operators is the same as \eqref{eq:general-Bs} after re-defining some operators, and the correspondences are summarized below.
$$
\begin{aligned}
\mathcal{B}_1 \bm{v} \text{ on } \mathcal{F}_h^0&: \llbracket \bm{v} \rrbracket_{\rm t} \mapsto [\bm{v}]_{\rm t}, 
\quad  \nabla_h \times \bm{v} \mapsto \nabla_h \cdot (\mathbf{R}\bm{v}), 
\quad \bm{\beta} \times \bm{v} \mapsto \mathbf{R}\bm{\beta} \cdot \bm{v},\\
\mathcal{B}_1^D \bm{v} \text{ on } \mathcal{F}_{h,D}^\partial &: \bm{n} \times \bm{v} \mapsto (\mathbf{R}\bm{n} \cdot \bm{v}) \mathbf{R}\bm{n}, 
\quad [\bm{n} \times (\nabla_h \times \bm{v})] \times \bm{n} \mapsto \nabla_h \cdot (\mathbf{R}\bm{v}) \mathbf{R}\bm{n}.
\end{aligned}
$$
In the same fashion, we can deduce the variational forms, for which the theory discussed in Sections \ref{sc:stability}-\ref{sc:error} still works.

\subsection{Experiment I: 3D smooth solution} We take $\bm{\beta}=(1,2,3)^T$, $\gamma=0$ so that the Assumption \ref{as:Friedrichs} (degenerate Friedrichs system) is valid with $\rho(x) = 0$.  For various diffusion coefficient $\varepsilon$ as $1,10^{-3},10^{-9}$, the forcing term $\bm{f}$ is chosen so that the analytical solution of (\ref{eq:Hcurl-cd}) is given by $\bm{u}(x,y,z)=(\sin(y),\sin(z),\sin(x))^T$. The boundary data is determined 
with the following boundary setting
$$
\Gamma_D:\{x=0 \text{ or } x=1 \text{ or } y=0 \text{ or } y=1\}, \quad
\Gamma_N:\{z=0 \text{ or } z=1\}.
$$
In the DG variational formulation \eqref{eq:variational-form}, we set $\eta=\tau=100$, $\theta=1$, $\alpha_d^\pm=\frac{1}{2}$, $\alpha^\pm = \frac{1+\text{sgn}(\bm{\beta}\cdot \bm{n}^\pm)}{2}$.

Figure \ref{fig:3dorder} represents, on a log-log scale, the convergence diagrams in the DG energy norm $\triplenorm{\cdot}$ and $L^2$ norm for $k=1$ versus the mesh size $h$. In terms of the DG energy norm \eqref{eq:energynorm}, we observe the first-order accuracy when diffusion dominates and order $3/2$ when convection dominates, which is in agreement with Theorem \ref{tm:estimate}. 
We also observe that, possibly due to the smoothness of the solution, the second-order convergence of $L^2$ norm is obtained in all regimes. 
\begin{figure}[!htbp]
	\centering
	\includegraphics[width=.3\textwidth]{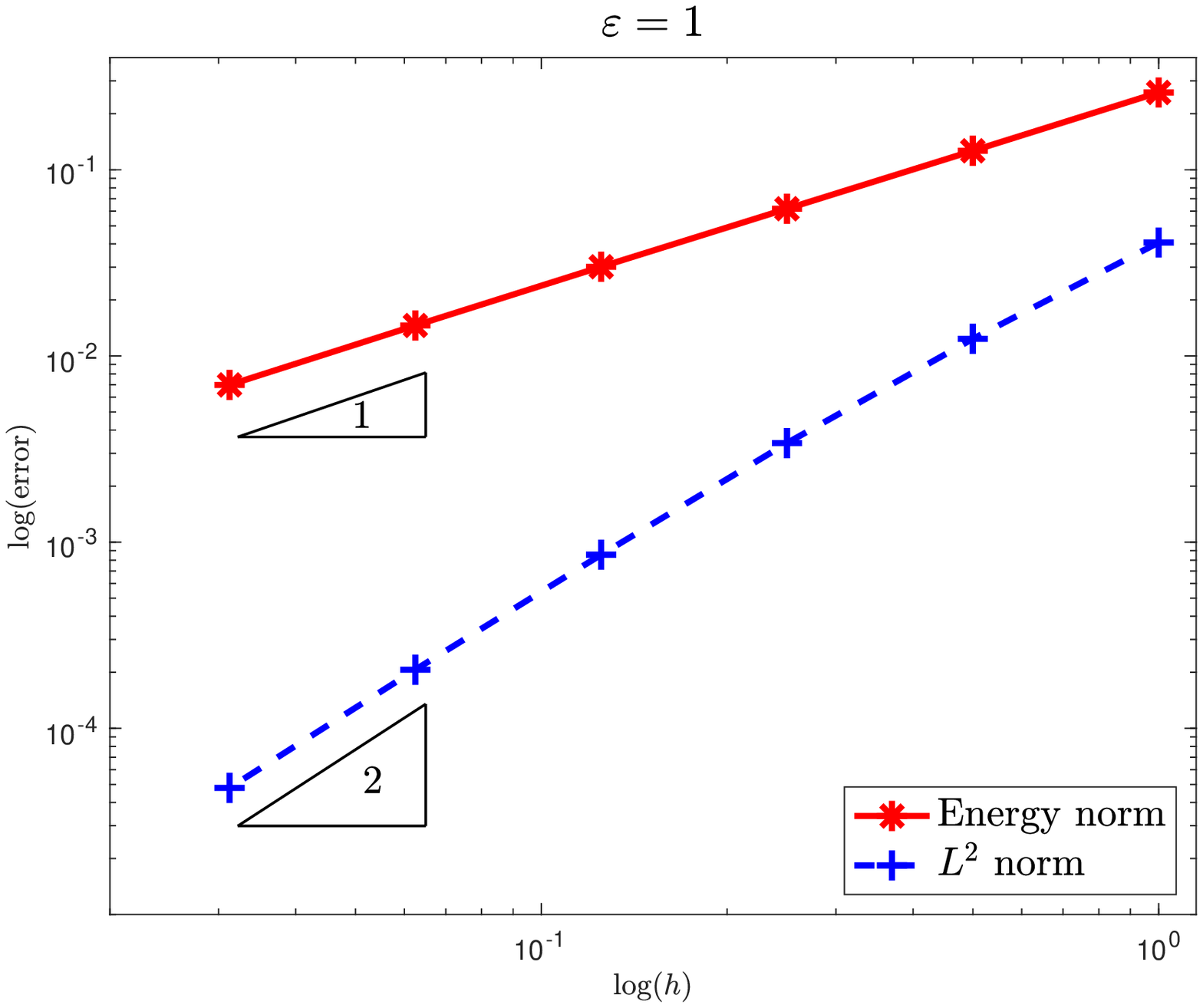}
	\includegraphics[width=.3\textwidth]{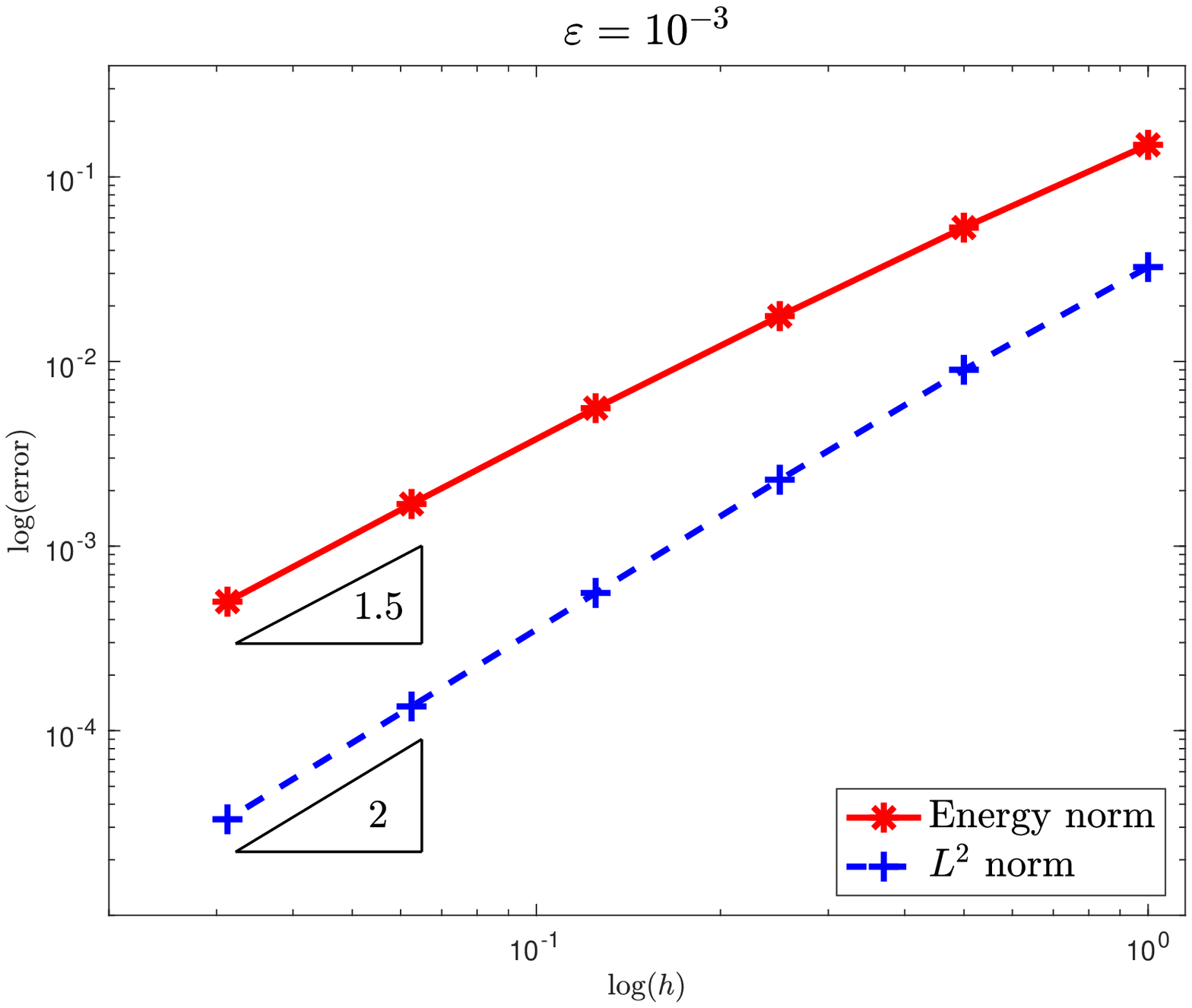}
	\includegraphics[width=.3\textwidth]{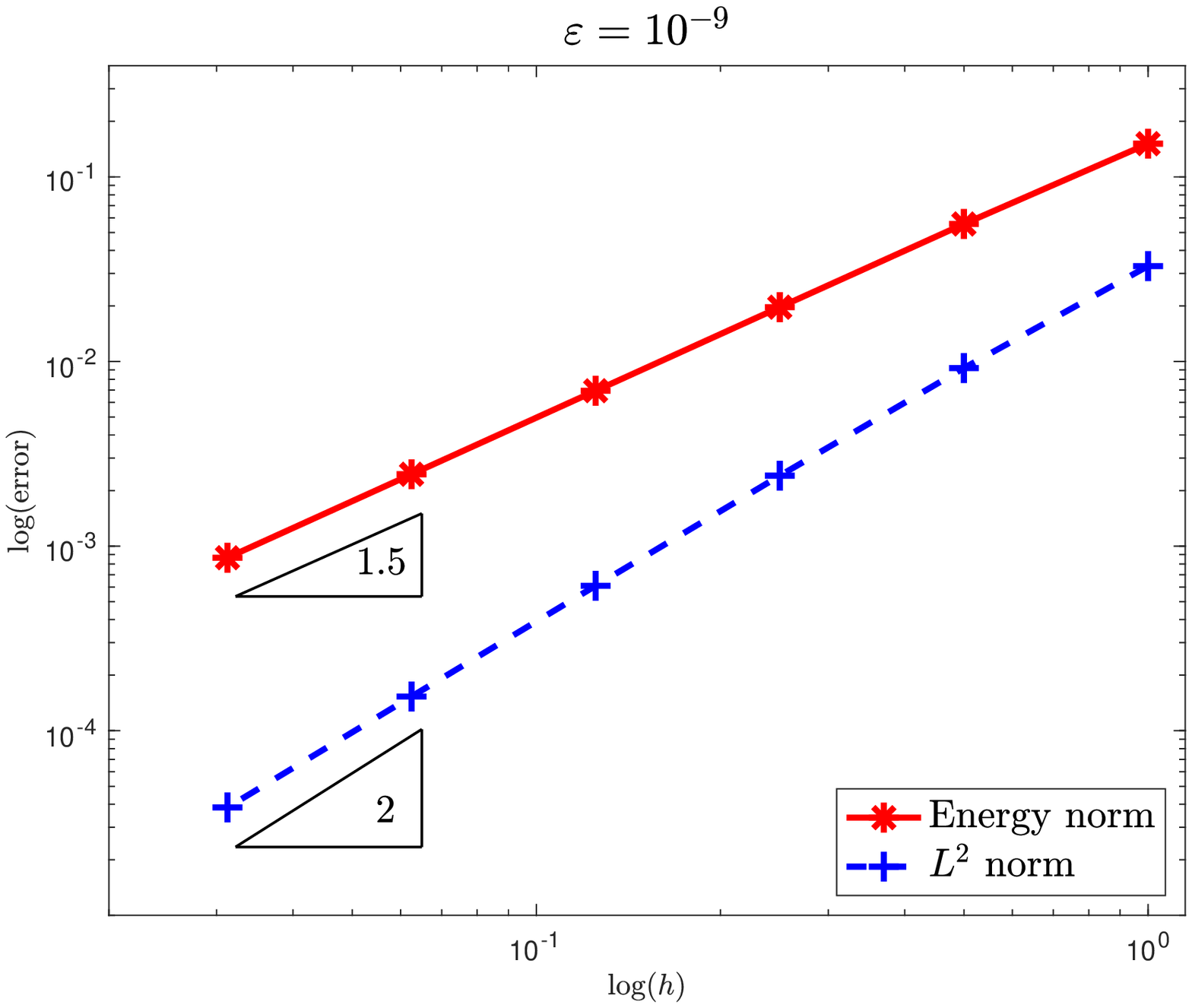}
	\caption{Experiment I. Convergence diagrams in the DG energy norm and $L^2$ norm for the 3D problem with a smooth solution.}\label{fig:3dorder}
\end{figure}

\subsection{Experiment II: 2D smooth solution} To verify the feasibility of the 2D problem \eqref{eq:Hcurl-cd-2d}, we take $\bm{\beta}=(1,1)^T$, $\gamma=0$, and vary the diffusion coefficient $\varepsilon$ as $1,10^{-3},10^{-9}$. The forcing term $\bm{f}$ and boundary data are chosen so that the analytical solution is given by
$\bm{u}(x,y)=(\sin(y),\sin(x))^T$,
with the following boundary setting
$$
\Gamma_D:\{y=0 \text{ or } y=1\},\quad
\Gamma_N:\{x=0 \text{ or } x=1\}.
$$
We set $\eta=\tau=10$, $\theta=1$, $\alpha_d^\pm=\frac{1}{2}$ in the DG scheme. For $\alpha$, we use a slightly different strategy such that  $\alpha^\pm = \frac{1+\text{sgn}(\bm{\beta}\cdot \bm{n}^\pm)}{2}$ for $\varepsilon\le10^{-3}$ and $\alpha^\pm = \frac{1+0.1\text{sgn}(\bm{\beta}\cdot \bm{n}^\pm)}{2}$ for $\varepsilon=1$ to obtain a better convergence result.

Figure \ref{fig:2dorder} shows  the convergence diagrams on a log-log scale in the DG energy norm $\triplenorm{\cdot}$ and $L^2$ norm for $k=1,2,$ versus the mesh size $h$. Similar to the 3D case, $k$-th order accuracy in the DG energy norm is observed when diffusion dominates, and order $k+1/2$ is observed when convection dominates. In all the tests above, the convergence orders in the $L^2$ norm seem to be $k+1$, which is out of the scope of current work. 

%The convergence rate of $L^2$ norm is more than $k+1/2$ in all the cases and in the convection dominated cases, the convergence rate is close to $k+1$.
\begin{figure}[!htbp]
	\centering
	\includegraphics[width=.3\textwidth]{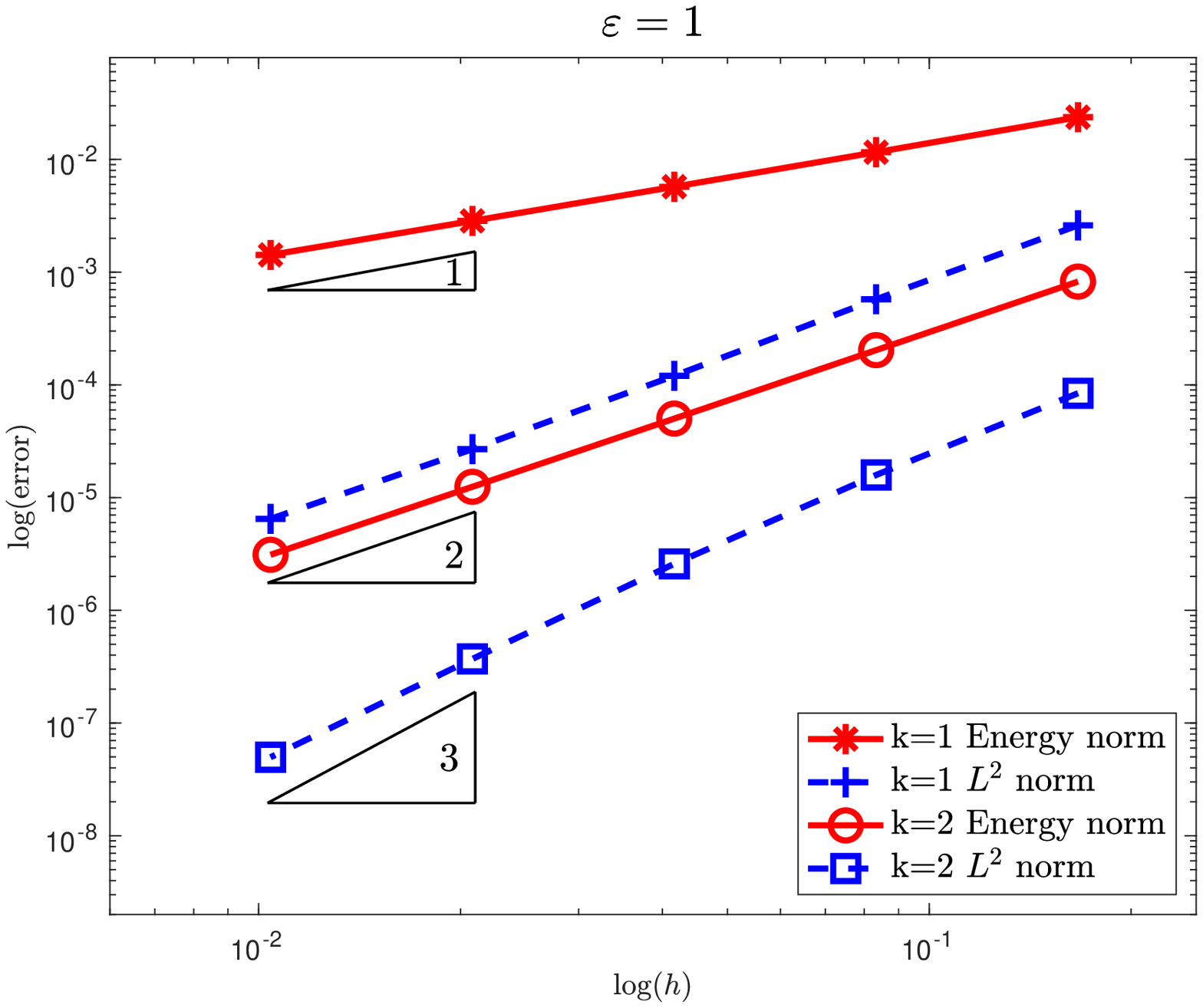}
	\includegraphics[width=.3\textwidth]{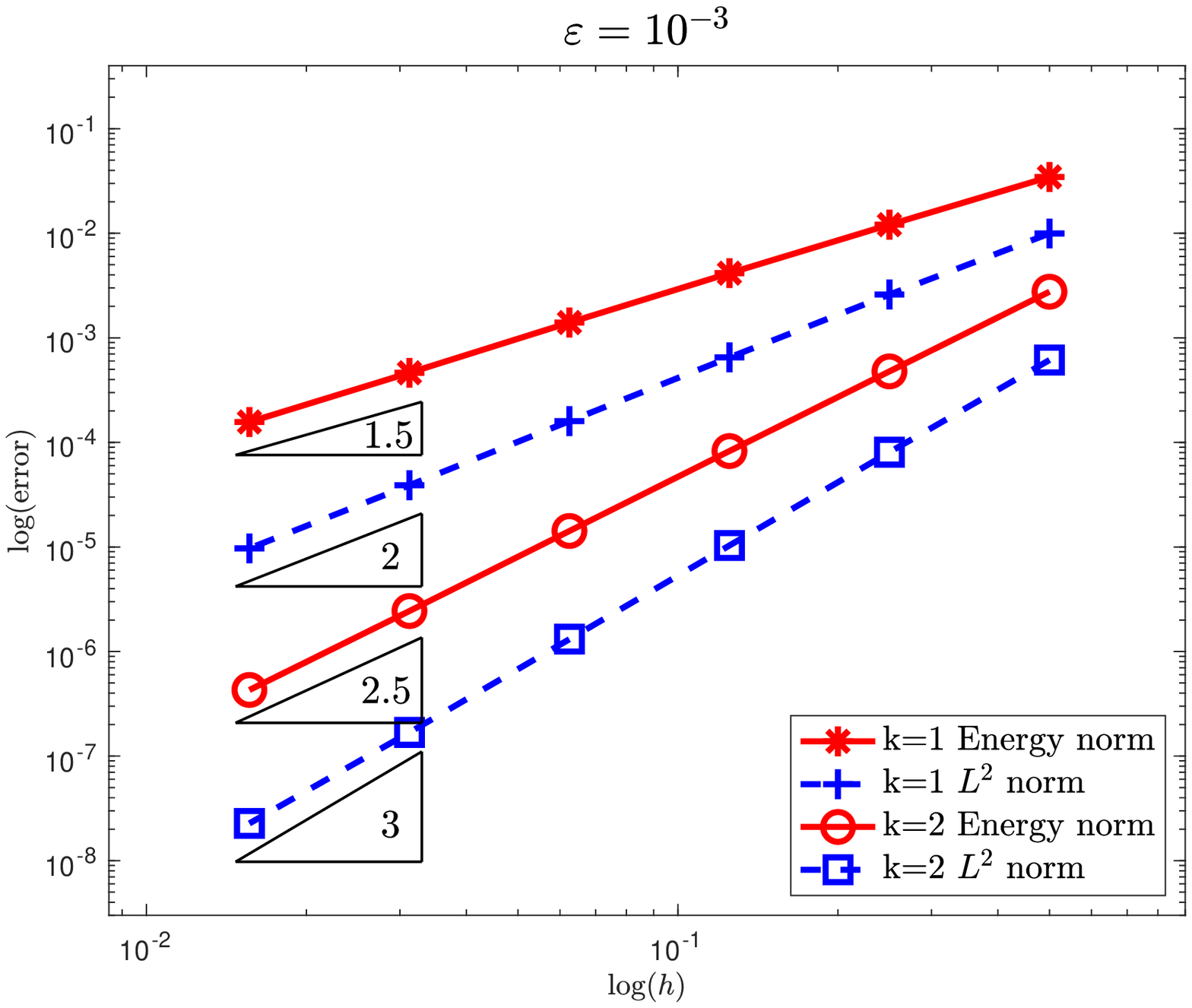}
	\includegraphics[width=.3\textwidth]{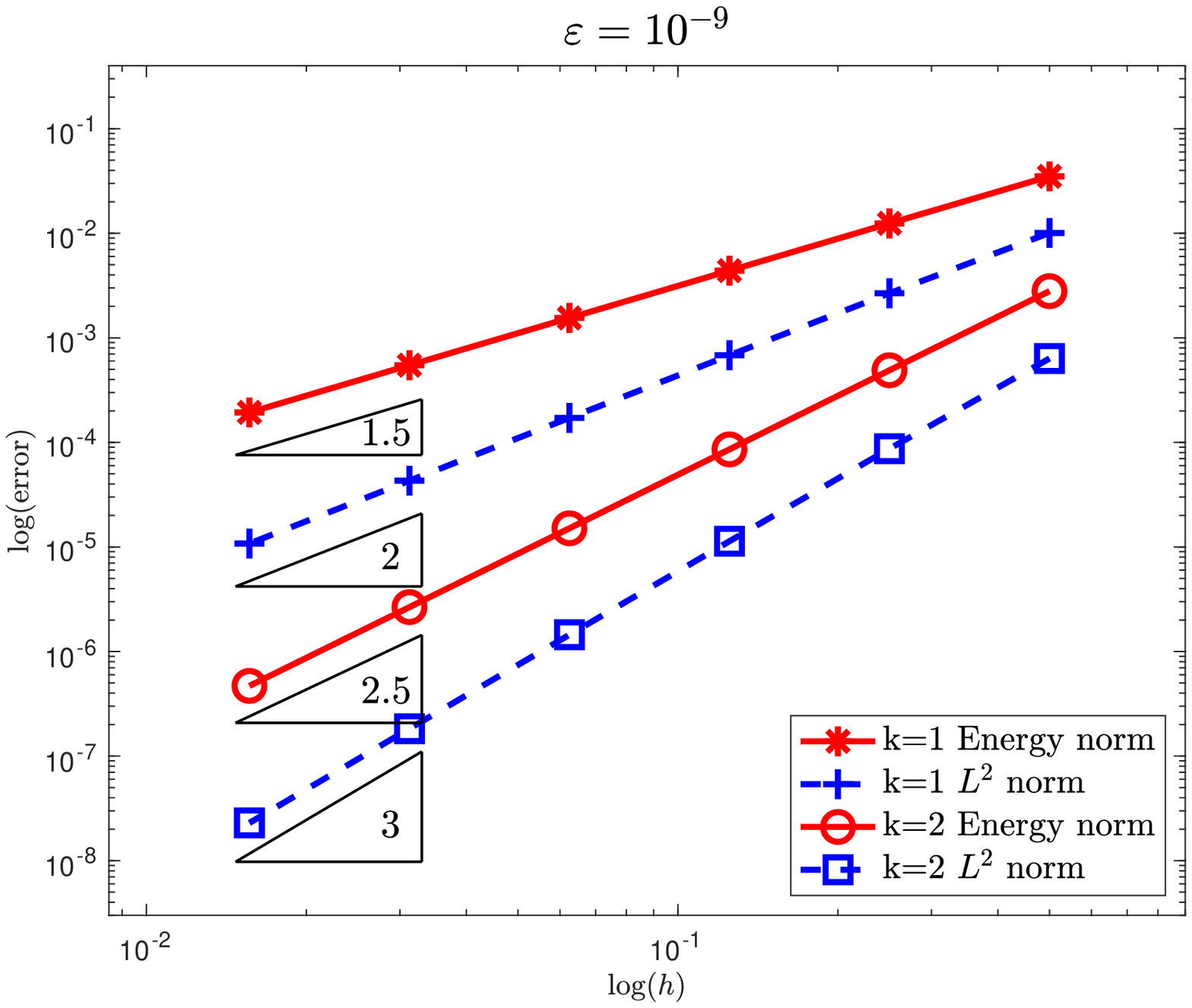}
	\caption{Experiment II. Convergence diagrams in the DG energy norm and $L^2$ norm for the 2D problem with a smooth solution.}\label{fig:2dorder}
\end{figure}

\subsection{Experiment III: Rotating flow} In this experiment, we compare the stabilization effect of different $\alpha$'s. The other parameters are set to be $\varepsilon=10^{-9}$, $\gamma=0$, $\bm{\beta}=[y-1/2,1/2-x]^T$, $\eta=\tau=10$, $\theta=1$ and $\alpha_d^\pm=\frac{1}{2}$. The solution is prescribed along the slit $1/2\times[0,1/2]$ as follows:
$$
\bm{u}(1/2,y)=(	\sin^2(2\pi y),\sin^2(2\pi y))^T, \quad y\in [0,1/2].
$$
In Figure \ref{fig:rot}, we plot the first component of approximation solution $u_1$ for various $\alpha$'s on a uniform mesh with $h=1/16$. As can be seen, in a stronger upwind setting, the continuity of the solution is improved and the oscillation is reduced. Meanwhile, the stronger upwind will lead to a larger $\alpha^{\max}$ that may pollute the constant in the error estimates.  That is, one needs to balance the numerical stability and accuracy when choosing $\alpha$.

\begin{figure}[!htbp]
	\centering
	\includegraphics[width=.3\textwidth]{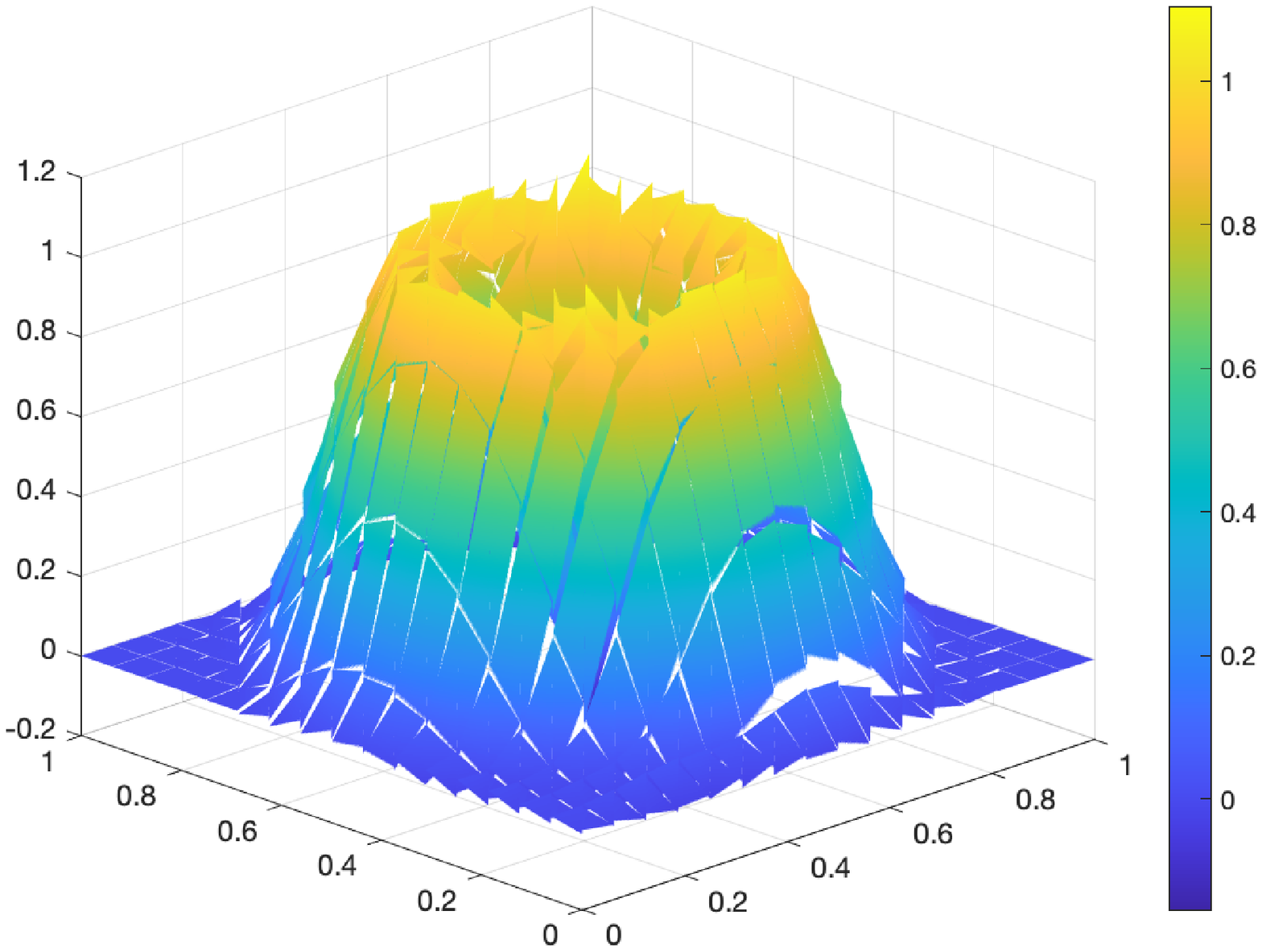}
	\includegraphics[width=.3\textwidth]{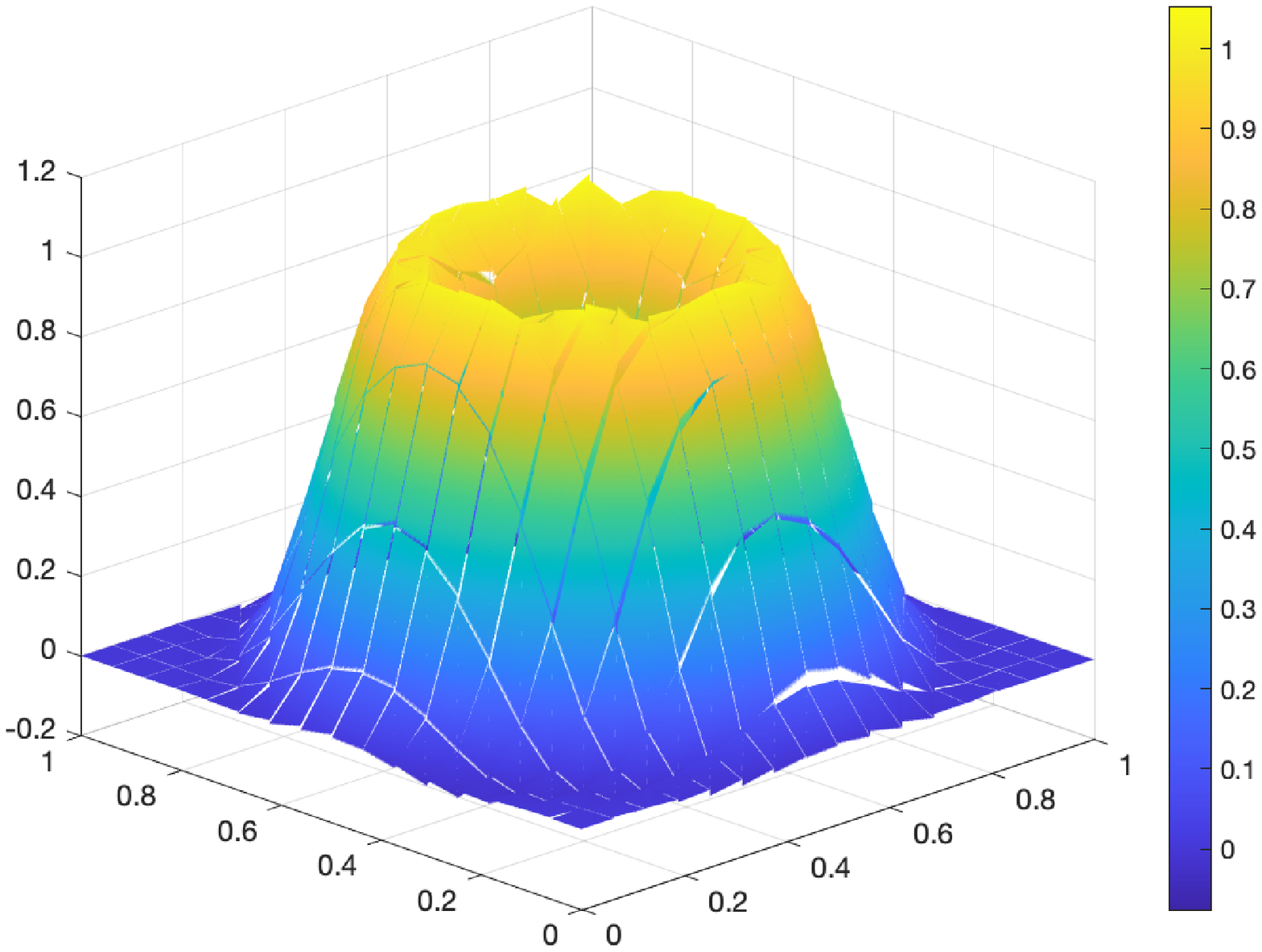}
	\includegraphics[width=.3\textwidth]{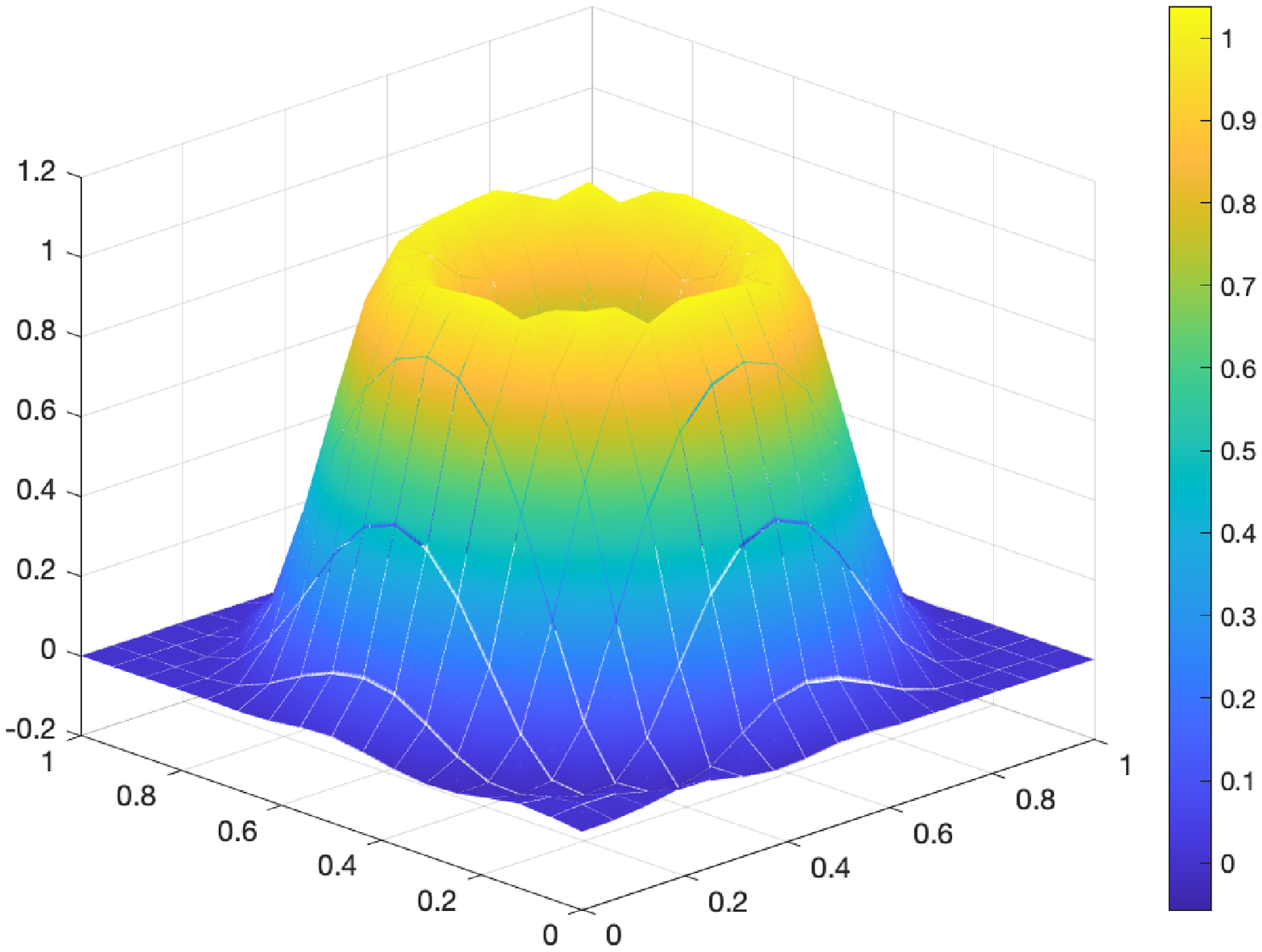}
	\caption{Experiment III. Approximation solution $u_1$ for $\varepsilon=10^{-9}$ with $h=1/16$ under weak upwind $\alpha^\pm=\frac{1+0.1{\rm sgn}(\bm{\beta}\cdot\bm{n}^\pm)}{2}$ (left), standard upwind $\alpha^\pm=\frac{1+{\rm sgn}(\bm{\beta}\cdot\bm{n}^\pm)}{2}$ (middle) and strong upwind $\alpha^\pm=\frac{1+10{\rm sgn}(\bm{\beta}\cdot\bm{n}^\pm)}{2}$ (right).}	\label{fig:rot}
\end{figure}

\subsection{Experiment IV: Internal layers}

The next example is devoted to assessing the performance of the DG method with different $\alpha_d$'s in the presence of interior layers. We set $\varepsilon=10^{-3}$, $\gamma=0$, $\bm{\beta}=[1/2,\sqrt{3}/2]^T$, $\bm{f}=(0,0)^T$, and Dirichlet boundary conditions for the following function:
$$
\bm{u}=\left\{
\begin{aligned}
	&(1,1)^T, \text{ on } \{y=0,0\le x\le 1\},\\
	&(1,1)^T, \text{ on } \{x=0,0\le y\le 0.2\},\\
	&(0,0)^T,\text{ elsewhere. }
\end{aligned}
 \right.
$$
In order to investigate the effect of $\alpha_d$, we consider the following two choices: $\alpha_{d}^\pm=\frac{1}{2}$ (standard) or $\alpha_d^\pm = \frac{1+\text{sgn}(\bm{\beta}\cdot \bm{n}^\pm)}{2}$ (upwind), under a weak upwind setting of $\alpha^\pm=\frac{1+0.1{\rm sgn}(\bm{\beta}\cdot\bm{n}^\pm)}{2}$ and $\tau=0.1$. In Figure \ref{fig:inlayer} , we plot the first component of the approximation solution. It is observed that the upwind setting of $\alpha_d$ induces a milder oscillation in the boundary area (in yellow) but performs a little bit worse in the smooth region. 
In this sense, we can use different strategies for $\alpha_d$ on different computation areas with some prior information to attain numerical solutions with better performance.
\begin{figure}[!htbp]
	\centering
	\includegraphics[width=.45\textwidth]{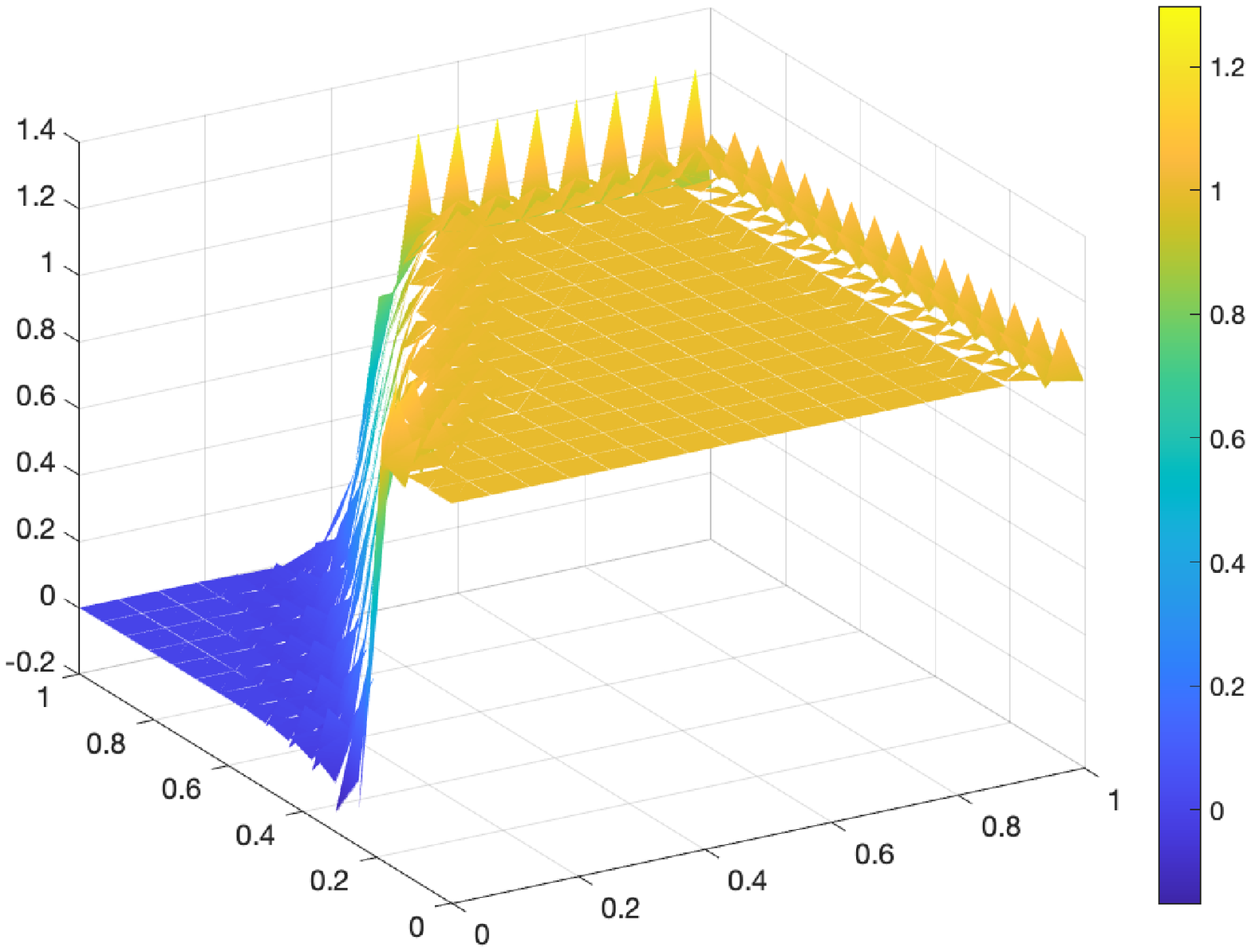}
	\includegraphics[width=.45\textwidth]{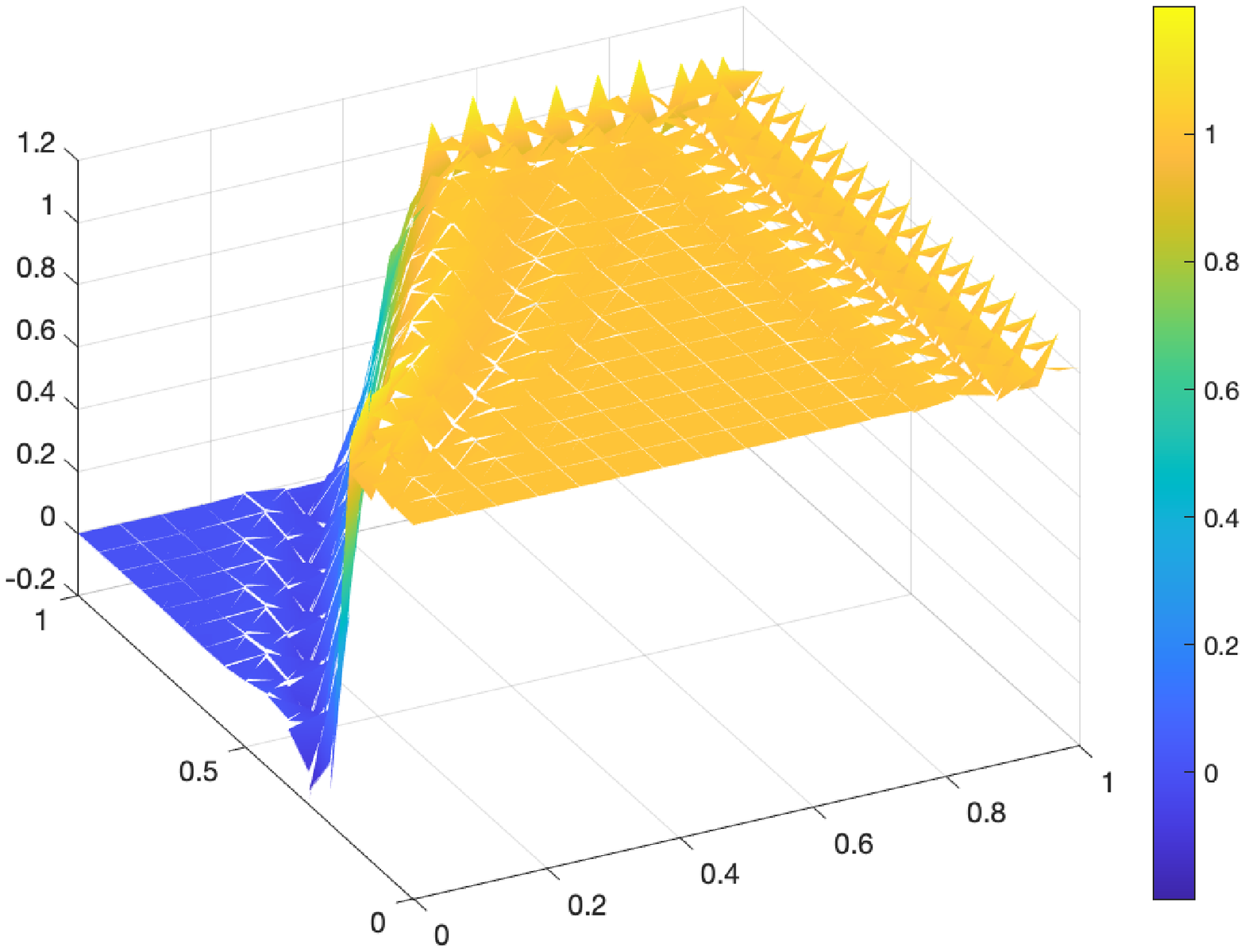}
	\caption{Experiment IV. Approximation solution $u_1$ for $\varepsilon=10^{-3}$ with $h=1/16$ under  standard average $\alpha^\pm_d=\frac{1}{2}$ (left) and upwind $\alpha^\pm_d=\frac{1+{\rm sgn}(\bm{\beta}\cdot\bm{n}^\pm)}{2}$ (right). }\label{fig:inlayer}
\end{figure}

\subsection{Experiment V: Boundary layers}
In this example, we apply the DG method to a boundary layer problem in 2D. The data are $\gamma=0$ and $\bm{\beta}=(1,2)^T$, and we vary the diffusion coefficient $\varepsilon$. The forcing term $\bm{f}$ is chosen to be $(1,1)^T$ and the homogeneous Dirichlet boundary condition is given.

As can be seen in Figure \ref{fig:bdy_layer}, the numerical solutions are stable in the sense that no spurious oscillation is observed with vanishing $\varepsilon$. Since the boundary conditions are imposed in a weak sense, the boundary layer is not captured by the DG approximations, which is also observed in the scalar case \cite{ayuso2009discontinuous}.
\begin{figure}
	\centering
	\includegraphics[width=.45\textwidth]{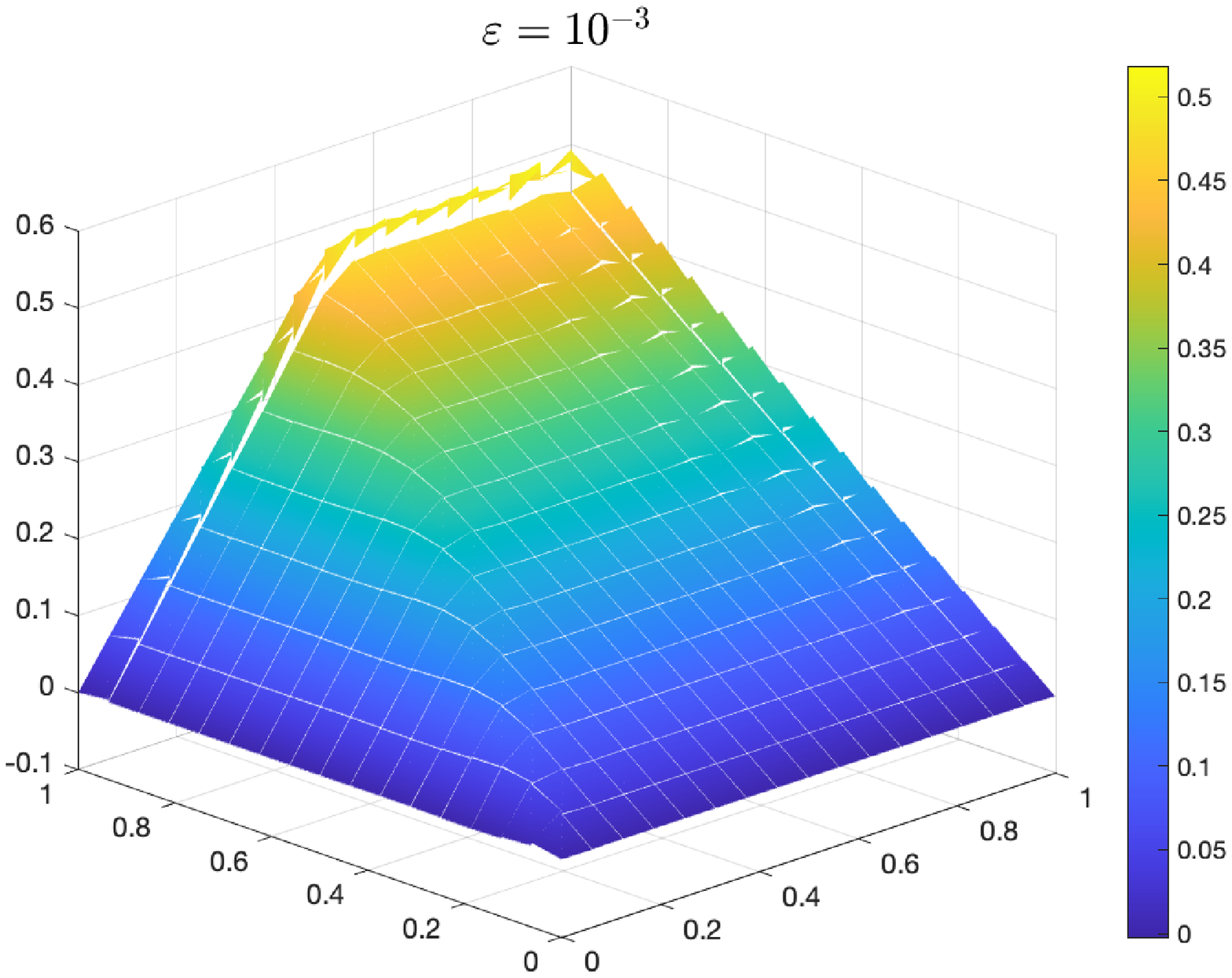}
	\includegraphics[width=.45\textwidth]{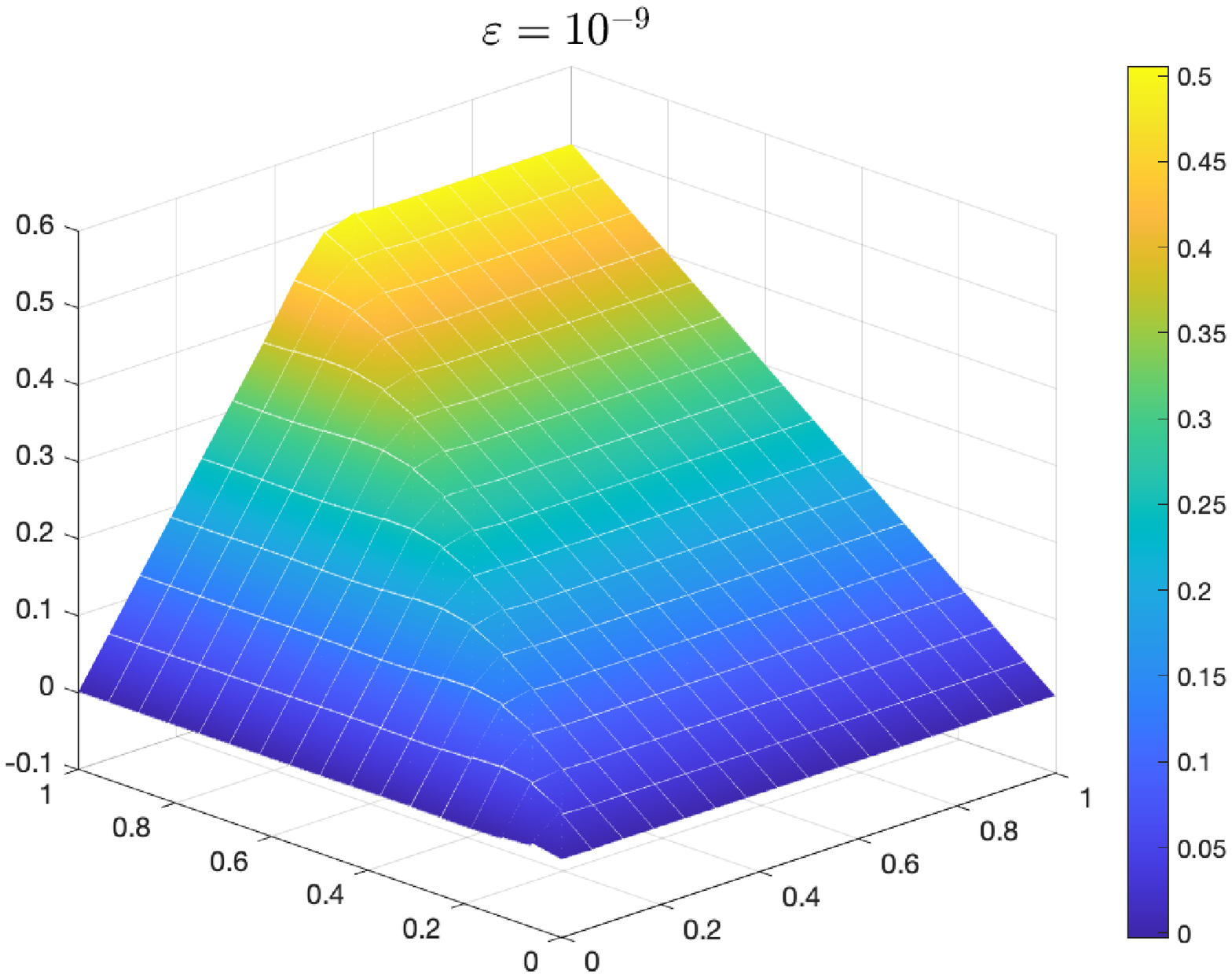}
	\caption{Experiment V. Approximation solution $u_1$ for $\varepsilon=10^{-3}$ (left) and  $\varepsilon=10^{-9}$ (right) with $h=1/16$ under standard upwind stabilization $\alpha^\pm=\frac{1+{\rm sgn}(\bm{\beta}\cdot\bm{n}^\pm)}{2}$ and $\alpha_d^\pm=\frac{1}{2}$ .}\label{fig:bdy_layer}
\end{figure}

%% end of file

\appendix
\section{Validation of Assumption \ref{as:normal-beta}} \label{as:normal-beta-proof}
We start from the case in which $\bm{\beta}=(\beta^1,\cdots,\beta^d)$ is a constant vector. For any element $T$, starting from an arbitrary vertex we have $d$ different edges of which the unit directions, denoted by $\bm{t}_1,\cdots,\bm{t}_d$, form a basis in $\mathbb{R}^d$. Therefore, $\bm{\beta}$ can be decomposed by this basis as $\bm{\beta}=\sum_{i=1}^{d}\alpha_i\bm{t}_i$ for some $\alpha_i$.  Without loss of generality, we assume that the $k$-th component of $\bm{\beta}$ obtains its infinity norm, i.e., $|\beta^k|=\|\bm{\beta}\|_{0,\infty,T}$ or 
$$
|\sum_{i=1}^{d}\alpha_it_i^k|=|\beta^k|=\|\bm{\beta}\|_{0,\infty,T},
$$
where $t_i^k$ is the $k$-th component of $\bm{t}_i$ and $t_i^k\le 1$ since $\bm{t}_i$ is a unit vector. Therefore, we can easily deduce that there exists at least one $\alpha_j$ such that $\alpha_j\ge \frac{\|\bm{\beta\|_{0,\infty,T}}}{d}$. Choose the facet $F$ that opposes the other vertex of the edge associated with $\bm{t}_j$ and there exists a positive constant $C$, depending only on the shape-regularity constant, such that 
$$
|\bm{n}_F\cdot \bm{t}_j |\ge C.
$$
Note that the edge of $\bm{t}_i ~(i\neq j)$ is on the facet $F$, we have $\bm{n}_F\cdot \bm{t_i}=0~(i\neq j)$, whence
$$
|\bm{\beta}\cdot \bm{n}_F|=|\sum_{i=1}^d\alpha_i\bm{t}_i\cdot \bm{n}_F|=|\alpha_j\bm{t}_j\cdot \bm{n}_F|\ge C\frac{\|\bm{\beta}\|_{0,\infty,T}}{d}=C\frac{\|\bm{\beta}\|_{0,\infty,F}}{d}.
$$
Therefore, the Assumption \ref{as:normal-beta} holds on the facet $F$.  For a variable $\bm{\beta}$, since $\bm{\beta} \neq \bm{0}$ and $\bm{\beta} \in \bm{W}^{1,\infty}(\Omega)$ from Assumption \ref{as:b0}, a simple perturbation argument leads to the Assumption \ref{as:normal-beta} for $h < h_0$, where $h_0 > 0$ is a fixed mesh size. 
\bibliographystyle{siamplain}
\bibliography{curl_DG.bib} 

\begin{thebibliography}{10}

\bibitem{arnold2002unified}
{\sc D.~N. Arnold, F.~Brezzi, B.~Cockburn, and L.~D. Marini}, {\em {Unified
  analysis of discontinuous Galerkin methods for elliptic problems}}, SIAM
  journal on numerical analysis, 39 (2002), pp.~1749--1779.

\bibitem{ayuso2009discontinuous}
{\sc B.~Ayuso and L.~D. Marini}, {\em Discontinuous {G}alerkin methods for
  advection-diffusion-reaction problems}, SIAM Journal on Numerical Analysis,
  47 (2009), pp.~1391--1420.

\bibitem{braack2006local}
{\sc M.~Braack and E.~Burman}, {\em {Local projection stabilization for the
  Oseen problem and its interpretation as a variational multiscale method}},
  SIAM Journal on Numerical Analysis, 43 (2006), pp.~2544--2566.

\bibitem{brezzi2006stabilization}
{\sc F.~Brezzi, B.~Cockburn, L.~D. Marini, and E.~S{\"u}li}, {\em Stabilization
  mechanisms in discontinuous {G}alerkin finite element methods}, Computer
  Methods in Applied Mechanics and Engineering, 195 (2006), pp.~3293--3310.

\bibitem{brezzi1998further}
{\sc F.~Brezzi, L.~P. Franca, and A.~Russo}, {\em Further considerations on
  residual-free bubbles for advective-diffusive equations}, Computer Methods in
  Applied Mechanics and Engineering, 166 (1998), pp.~25--33.

\bibitem{brezzi1999priori}
{\sc F.~Brezzi, T.~J. Hughes, L.~Marini, A.~Russo, and E.~S{\"u}li}, {\em A
  priori error analysis of residual-free bubbles for advection-diffusion
  problems}, SIAM journal on numerical analysis, 36 (1999), pp.~1933--1948.

\bibitem{brezzi1998applications}
{\sc F.~Brezzi, D.~Marini, and A.~Russo}, {\em Applications of the pseudo
  residual-free bubbles to the stabilization of convection-diffusion problems},
  Computer Methods in Applied Mechanics and Engineering, 166 (1998),
  pp.~51--63.

\bibitem{brezzi2004discontinuous}
{\sc F.~Brezzi, L.~D. Marini, and E.~S{\"u}li}, {\em {Discontinuous Galerkin
  methods for first-order hyperbolic problems}}, Mathematical Models and
  Methods in Applied Sciences, 14 (2004), pp.~1893--1903.

\bibitem{brezzi1994choosing}
{\sc F.~Brezzi and A.~Russo}, {\em Choosing bubbles for advection-diffusion
  problems}, Mathematical Models and Methods in Applied Sciences, 4 (1994),
  pp.~571--587.

\bibitem{burman2005unified}
{\sc E.~Burman}, {\em A unified analysis for conforming and nonconforming
  stabilized finite element methods using interior penalty}, SIAM Journal on
  Numerical Analysis, 43 (2005), pp.~2012--2033.

\bibitem{burman2010consistent}
{\sc E.~Burman}, {\em Consistent {SUPG}-method for transient transport
  problems: {S}tability and convergence}, Computer Methods in Applied Mechanics
  and Engineering, 199 (2010), pp.~1114--1123.

\bibitem{burman2007continuous}
{\sc E.~Burman and A.~Ern}, {\em Continuous interior penalty $hp$-finite
  element methods for advection and advection-diffusion equations}, Mathematics
  of Computation, 76 (2007), pp.~1119--1140.

\bibitem{burman2009weighted}
{\sc E.~Burman, J.~Guzm{\'a}n, and D.~Leykekhman}, {\em Weighted error
  estimates of the continuous interior penalty method for singularly perturbed
  problems}, IMA journal of numerical analysis, 29 (2009), pp.~284--314.

\bibitem{burman2004edge}
{\sc E.~Burman and P.~Hansbo}, {\em Edge stabilization for {G}alerkin
  approximations of convection-diffusion-reaction problems}, Computer Methods
  in Applied Mechanics and Engineering, 193 (2004), pp.~1437--1453.

\bibitem{castillo2002optimal}
{\sc P.~Castillo, B.~Cockburn, D.~Sch{\"o}tzau, and C.~Schwab}, {\em Optimal a
  priori error estimates for the hp-version of the local discontinuous
  {G}alerkin method for convection--diffusion problems}, Mathematics of
  computation, 71 (2002), pp.~455--478.

\bibitem{chen2005optimal}
{\sc L.~Chen and J.~Xu}, {\em An optimal streamline diffusion finite element
  method for a singularly perturbed problem}, Contemporary Mathematics, 383
  (2005), p.~191.

\bibitem{cockburn1999some}
{\sc B.~Cockburn and C.~Dawson}, {\em Some extensions of the local
  discontinuous {G}alerkin method for convection-diffusion equations in
  multidimensions},  (1999).

\bibitem{cockburn2008optimal}
{\sc B.~Cockburn, B.~Dong, and J.~Guzm{\'a}n}, {\em Optimal convergence of the
  original {DG} method for the transport-reaction equation on special meshes},
  SIAM Journal on Numerical Analysis, 46 (2008), pp.~1250--1265.

\bibitem{devinatz1974asymptotic}
{\sc A.~Devinatz, R.~Ellis, and A.~Friedman}, {\em {The asymptotic behavior of
  the first real eigenvalue of second order elliptic operators with a small
  parameter in the highest derivatives, II}}, Indiana University Mathematics
  Journal, 23 (1974), pp.~991--1011.

\bibitem{dorfler1999uniform2}
{\sc W.~D{\"o}rfler}, {\em Uniform a priori estimates for singularly perturbed
  elliptic equations in multidimensions}, SIAM journal on numerical analysis,
  36 (1999), pp.~1878--1900.

\bibitem{dorfler1999uniform}
{\sc W.~D{\"o}rfler}, {\em Uniform error estimates for an exponentially fitted
  finite element method for singularly perturbed elliptic equations}, SIAM
  journal on numerical analysis, 36 (1999), pp.~1709--1738.

\bibitem{ern2006discontinuous}
{\sc A.~Ern and J.-L. Guermond}, {\em Discontinuous {G}alerkin methods for
  {F}riedrichs' systems. i. general theory}, SIAM journal on numerical
  analysis, 44 (2006), pp.~753--778.

\bibitem{franca1992stabilized}
{\sc L.~P. Franca, S.~L. Frey, and T.~J. Hughes}, {\em Stabilized finite
  element methods: {I}. {A}pplication to the advective-diffusive model},
  Computer Methods in Applied Mechanics and Engineering, 95 (1992),
  pp.~253--276.

\bibitem{franca2002stability}
{\sc L.~P. Franca and L.~Tobiska}, {\em Stability of the residual free bubble
  method for bilinear finite elements on rectangular grids}, IMA Journal of
  Numerical Analysis, 22 (2002), pp.~73--87.

\bibitem{friedrichs1958symmetric}
{\sc K.~O. Friedrichs}, {\em Symmetric positive linear differential equations},
  Communications on Pure and Applied Mathematics, 11 (1958), pp.~333--418.

\bibitem{fu2015analysis}
{\sc G.~Fu, W.~Qiu, and W.~Zhang}, {\em An analysis of {HDG} methods for
  convection-dominated diffusion problems}, ESAIM: Mathematical Modelling and
  Numerical Analysis, 49 (2015), pp.~225--256.

\bibitem{gerbeau2006mathematical}
{\sc J.-F. Gerbeau, C.~Le~Bris, and T.~Leli{\`e}vre}, {\em Mathematical methods
  for the magnetohydrodynamics of liquid metals}, Clarendon Press, 2006.

\bibitem{heumann2011eulerian}
{\sc H.~Heumann and R.~Hiptmair}, {\em {Eulerian and semi-Lagrangian methods
  for convection-diffusion for differential forms}}, Discrete Continuous
  Dynamical Systems, 29 (2011), pp.~1497--1516.

\bibitem{heumann2013stabilized}
{\sc H.~Heumann and R.~Hiptmair}, {\em Stabilized {G}alerkin methods for
  magnetic advection}, ESAIM: Mathematical Modelling and Numerical Analysis, 47
  (2013), pp.~1713--1732.

\bibitem{heumann2015stabilized}
{\sc H.~Heumann, R.~Hiptmair, and C.~Pagliantini}, {\em Stabilized {G}alerkin
  for transient advection of differential forms}, PhD thesis, SAM, ETH
  Z{\"u}rich, 2015.

\bibitem{houston2002discontinuous}
{\sc P.~Houston, C.~Schwab, and E.~S{\"u}li}, {\em Discontinuous hp-finite
  element methods for advection-diffusion-reaction problems}, SIAM Journal on
  Numerical Analysis, 39 (2002), pp.~2133--2163.

\bibitem{hughes1989new}
{\sc T.~J. Hughes, L.~P. Franca, and G.~M. Hulbert}, {\em {A new finite element
  formulation for computational fluid dynamics: VIII. The
  Galerkin/least-squares method for advective-diffusive equations}}, Computer
  methods in applied mechanics and engineering, 73 (1989), pp.~173--189.

\bibitem{hughes1979finite}
{\sc T.~J. Hughes, W.~K. Liu, and A.~Brooks}, {\em Finite element analysis of
  incompressible viscous flows by the penalty function formulation}, Journal of
  computational physics, 30 (1979), pp.~1--60.

\bibitem{matthies2007unified}
{\sc G.~Matthies, P.~Skrzypacz, and L.~Tobiska}, {\em {A unified convergence
  analysis for local projection stabilisations applied to the Oseen problem}},
  ESAIM: Mathematical Modelling and Numerical Analysis, 41 (2007),
  pp.~713--742.

\bibitem{mizukami1985petrov}
{\sc A.~Mizukami and T.~J. Hughes}, {\em {A Petrov-Galerkin finite element
  method for convection-dominated flows: an accurate upwinding technique for
  satisfying the maximum principle}}, Computer methods in applied mechanics and
  engineering, 50 (1985), pp.~181--193.

\bibitem{o1991analysis}
{\sc E.~O'Riordan and M.~Stynes}, {\em An analysis of some exponentially fitted
  finite element methods for singularly perturbed elliptic problems},
  Computational methods for boundary and interior layers in several dimensions,
  1 (1991), pp.~138--153.

\bibitem{o1991globally}
{\sc E.~O’Riordan and M.~Stynes}, {\em A globally uniformly convergent finite
  element method for a singularly perturbed elliptic problem in two
  dimensions}, Mathematics of computation, 57 (1991), pp.~47--62.

\bibitem{perugia2003hp}
{\sc I.~Perugia and D.~Sch{\"o}tzau}, {\em {The $hp$-local discontinuous
  Galerkin method for low-frequency time-harmonic Maxwell equations}},
  Mathematics of Computation, 72 (2003), pp.~1179--1214.

\bibitem{sun2005symmetric}
{\sc S.~Sun and M.~F. Wheeler}, {\em {Symmetric and nonsymmetric discontinuous
  Galerkin methods for reactive transport in porous media}}, SIAM Journal on
  Numerical Analysis, 43 (2005), pp.~195--219.

\bibitem{wang1997novel}
{\sc S.~Wang}, {\em A novel exponentially fitted triangular finite element
  method for an advection-diffusion problem with boundary layers}, Journal of
  Computational Physics, 134 (1997), pp.~253--260.

\bibitem{wu2020simplex}
{\sc S.~Wu and J.~Xu}, {\em Simplex-averaged finite element methods for
  {$H({\rm grad})$}, {$H({\rm curl})$}, and {$H({\rm div})$}
  convection-diffusion problems}, SIAM Journal on Numerical Analysis, 58
  (2020), pp.~884--906.

\bibitem{xu1999monotone}
{\sc J.~Xu and L.~Zikatanov}, {\em A monotone finite element scheme for
  convection-diffusion equations}, Mathematics of Computation, 68 (1999),
  pp.~1429--1446.

\bibitem{zarin2005interior}
{\sc H.~Zarin and H.-G. Roos}, {\em Interior penalty discontinuous
  approximations of convection-diffusion problems with parabolic layers},
  Numerische Mathematik, 100 (2005), pp.~735--759.

\end{thebibliography}

\end{document}